\documentclass{article}[12pt]
\usepackage{natbib}
\usepackage[pdftex,pdfborder={0 0 0}, 
colorlinks=true, 
linkcolor=blue, 
citecolor=blue, 
pagebackref=false, 
]{hyperref}
\usepackage[utf8]{inputenc}
\usepackage{amsmath,amsfonts,amsthm}
\usepackage{xcolor}
\usepackage{graphicx}
\usepackage{todonotes}
\newtheorem{proposition}{Proposition}
\newtheorem{theorem}{Theorem}
\newtheorem{lemma}{Lemma}
\theoremstyle{definition}
\newtheorem{definition}{Definition}
\title{A mean-field game model of electricity market dynamics\footnote{We thank the anonymous reviewer for insightful comments on a previous version of this paper, and our research assistant Shiqi Lui for providing an implementaiton of the model with nonzero plant construction time. The research of Peter Tankov was supported by the FIME (Finance for Energy Markets) research initiative of the Institut Europlace de Finance. The research of Peter Tankov and Alicia Bassière was supported by the ANR (project EcoREES ANR-19-CE05-0042).}}
\author{Alicia Bassière\footnote{CREST, ENSAE, Institut Polytechnique de Paris}, Roxana Dumitrescu\footnote{King's College London} and Peter Tankov$^{\dag}$\footnote{Corresponding author, email: \texttt{peter.tankov@ensae.fr}}}
\date{}

\begin{document}

\maketitle

\begin{abstract}
We develop a model for the long-term dynamics of electricity market, based on mean-field games of optimal stopping. Our paper extends the recent contribution [Aïd, René, Roxana Dumitrescu, and Peter Tankov, ``The entry and exit game in the electricity markets: A mean-field game approach." Journal of Dynamics \& Games 8.4 (2021): 331] in several ways, making the model much more realistic, especially for describing the medium-term impacts of energy transition on electricity markets. 
In particular, we allow for an arbitrary number of technologies with endogenous fuel prices and enable the agents to both invest and divest. This makes it possible to describe the role of gas generation as a medium-term substitute for coal, to be replaced by renewable generation in the long term, and enables us to model the events like the 2022 energy price crisis. 
\end{abstract}


\section{Introduction}
The environmental transition of the world electricity industry is now well underway. According to the International Energy Agency (IEA), global wind and solar energy capacity increased by a factor of four between 2010 and 2020, reaching 8.4\% of the total electricity generation, while the share of coal-based electricity generation declined from 40.1\% to 35.1\% of the total over the same period, and this decline is bound to accelerate \citep{agency_oecd_2022}. However, the increasing share of renewable generation, and the closure of coal-fired power plants has also led to unintended consequences. On the one hand, low marginal costs of electricity generation push down the baseload electricity prices, eroding the profits of conventional producers, which are important for system stability. The capacity markets, developed in the recent years, aim to solve this problem by providing an additional stable source of revenues to conventional producers even when they are not actually generating electricity. On the other hand, high renewable penetration and the departure of coal increases the volatility of peakload prices, and makes the system more reliant on natural gas, which is often used as a substitute for coal due both to its lower environmental impact, and high ramping capacity of gas-fired power plants. As a result, an exogenous supply shock on gas prices may lead to a global energy price crisis sending shock waves across various sectors of industry, as we have witnessed in 2022.

The aim of this paper is to build a long-term model for the dynamics of the electricity industry, able to describe the energy transition and in particular to take into account the role of gas as a medium-term substitute for coal. Rather than using the ``central-planner" approach and looking for the cost-minimizing configuration, we assume that the dynamics of the electricity industry is the result of interaction of a large number of agents, each having its own objective, and look for the Nash equilibrium. Aiming for mathematical tractability, we adopt the technology of mean-field games, whereby the electricity industry is composed of a large number of small agents, and each agent is assumed to be a price-taker. 

In our model, there are $N$ types of conventional electricity producers (gas, coal, etc.) and $\overline N$ types of renewable electricity producers (solar, wind, etc.). Each conventional generator uses one of $K$ fuels (coal, gas, etc.) and pays a carbon price proportional to the emission intensity of the fuel used. The carbon price is exogenous and the fuel prices are determined endogenously, from an exogenous supply function and the aggregate demand of all producers using a given fuel. In addition to fuel costs and carbon price costs, each conventional producer faces idiosyncratic random costs, and determines the amount to offer in the market to maximize the instantaneous gain. Renewable producers, on the other hand, have a random capacity factor, do not pay fuel or carbon price costs and offer their entire (random) production in the market. 

The peak and off-peak electricity prices are determined from the renewable and conventional supply and an exogenous demand process through the merit order mechanism. Due to the mean-field game assumption, fuel and electricity price trajectories are deterministic in our model: there is no common noise. 

Each given class of producers may either contain the producers who are already in the process of constructing or operating the plant, or potential generation projects, for which the decision to enter the market is yet to be taken. These potential producers look for the optimal moment to enter the market (whereupon they move to the first category), while the producers who are already in the market look for the optimal time to exit. The optimal decision is taken by maximizing a gain functional, which depends on the future electricity and fuel price trajectories. We then look for the \emph{equilibrium} price trajectories: such that every agent follows the optimal entry/exit strategy and the fuel and electricity markets clear for the resulting distribution of agents at all times. Under suitable assumptions, we prove existence of equilibrium and uniqueness of the equilibrium price trajectories of electricity and fuels. An efficient fixed-point algorithm is then presented for computing the equilibrium. As a result of running the algorithm, we obtain the trajectories of the electricity price, fuel prices, installed generation capacity, and energy mix, for given trajectories of electricity demand and carbon price, and for given technology costs. The model allows to study the progressive substitution of different technologies (gas, coal, renewable) in the process of energy transition and the resulting effects on prices.
Modifying the parameters, we may compare the effect of different market design elements on electricity and fuel prices, renewable penetration and GHG emissions.

Our paper extends an earlier contribution \citep{aid2020entry} in a number of key directions. First, while \cite{aid2020entry} only considered two types of producers (gas and wind), we allow for an arbitrary number of producer types. More importantly, each producer who is not yet on the market may first decide to invest and then divest at a later date. Fuel prices are endogenously determined, while they were assumed to be exogenous in \citep{aid2020entry}. Finally, temporal aspects (such as finite plant age, finite construction time and realistic distribution of ages of plants already present in the market) are introduced in this paper. 

The purpose of this paper is to introduce the model and the numerical computation scheme, and illustrate its features on a stylized example. Detailed calibration of model features and parameters, as well as the analysis of results, will be addressed elsewhere. 



The paper is structured as follows. In the remaining part of the introduction we review the relevant literature. Section \ref{model.sec} presents our model of the electricity market. The relaxed solution approach and the main theoretical results are exposed in Section \ref{mfg.sec}. The numerical computation of the MFG solution and the examples are discussed in Section \ref{numerics.sec}. The proof of the main existence and uniqueness theorem is provided in the appendix.

\subsection{Literature review}

According to \citet{ventosa2005}, electricity market models can be classified into three main types: optimization models, market equilibrium models, and simulation models. We adopt this classification scheme to frame our literature review.

\paragraph{Optimization models}

Optimization models are widely used in operations research and engineering due to the availability of numerous solution algorithms that can easily incorporate technical constraints into the model. Typically, a representative company is modeled as a single agent whose objective is to minimize the cost of electricity production or maximize its profit. Technical constraints related to adequacy, storage, grid balancing, Kirchhoff laws, and generation units can be effectively integrated into such models. As a result, optimization models are particularly useful for addressing short-term problems.

Traditionally, price formation is considered exogenous or a function of the quantity offered by the optimizing agent. However, this type of model is generally unsuitable for representing competitive behavior, as it oversimplifies the strategic behavior of agents and behavioral biases such as risk aversion. Recent literature is increasingly trying to address this criticism \citep{petitet2017capacity}. Despite this, due to their mathematical simplicity, these models, which generally involve a linear problem (LP) or a mixed integer linear problem (MILP),  can be adapted for both deterministic and probabilistic approaches.

\paragraph{Market equilibrium models}

Market equilibrium models are widely used to analyze market mechanisms, market power, and public policies. These models incorporate all agents in competition, resulting in complex calculations, but allowing for strategic interactions to be represented. Equilibrium is formalized by a set of symmetric equations, where market participants generally aim to maximize their expected profits under market clearing conditions. Equilibrium models are particularly useful for medium-term problems, such as analyzing market entry dynamics.

Given their structure, equilibrium models are well-suited for introducing game theory concepts into the analysis of the electricity market. Common examples found in the literature include Cournot games, which involve competition in quantities, and Supply Function Equilibrium (SFE), which involves competition by bidding complete supply functions.

Cournot games have long been employed in the analysis of market power \citep[e.g.,][]{borenstein1999empirical}. In a Cournot equilibrium, firms simultaneously determine their output levels while taking their competitors' output as given. Each firm assumes that its rivals will produce a fixed quantity and then selects its output level to maximize its profits based on this assumption. However, the price predictions generated by the Cournot equilibrium are challenging to validate in practice \citep{baldick2004theory}.

SFE is a more realistic approach as it allows the introduction of technologies with different production and supply functions. SFE analysis has also been shown to be less sensitive to model parameters \citep{neuhoff2005network, willems2009cournot}. However, its greater mathematical complexity makes it unsustainable for a detailed market representation. More importantly, SFE models often have multiple equilibria, which can lead to unstable solutions \citep{willems2009cournot}.

The first example of SFE applied to the electricity market is \citet{green1992competition}. Another key difference between Cournot equilibrium and supply function equilibrium is that Cournot equilibrium assumes that firms are quantity setters, while supply function equilibrium assumes that firms are price takers. In other words, in a Cournot equilibrium, each firm sets its output level based on the expected market price, whereas in a supply function equilibrium, each firm sets its supply function based on the market price set by its competitors.


\paragraph{Simulation models}

Simulation models are increasingly being used in economics as an alternative to equilibrium models, especially when the latter become difficult to solve using numerical tools and machine learning. Instead of calculating an equilibrium, simulation models represent behavior through a set of sequential rules. These models are highly flexible and can represent all agents as well as technical constraints found in optimization models. Moreover, simulation models are a useful tool for studying repeated interactions and introducing uncertainty (for example, \citet{petitet2017capacity} provide an excellent illustration of this).

The simulation model category encompasses agent-based models (ABMs), which are increasingly popular for dynamic market analysis. In ABMs, agents are modeled as goal-oriented and adaptive, improving their decision-making through interactions and adapting to changes in the environment. However, these models can be challenging to replicate and are based on case-specific assumptions rather than fully tractable sets of equations. As such, they are not well-suited for generalizable studies and should be used with caution. Additionally, ABMs are difficult to validate empirically and are highly sensitive to initial conditions \citep{weidlich2008critical}. Moreover, ABMs often lack qualitative insights into agents' behavior.

\paragraph{Entry and exit models for electricity market}

The dynamics of market entry have been a subject of research since the opening to competition in the 1990s. Newbery's Supply Function Equilibrium model was extended by \citet{newbery1998competition} to include entry into the electricity market. While \citet{chronopoulos2014duopolistic} investigated the dynamics of entry under risk aversion and uncertainty, their analysis was limited to a duopoly. Similarly, \citet{goto2008entry} and \citet{takashima2008entry} studied the strategic entry of firms under price uncertainty and competition, but only with two plants. This limitation arises due to the computational burden of finding an equilibrium with a high number of players. Even when such an equilibrium is found, many strong assumptions are needed to ensure uniqueness, such as agent symmetry \citep{allcott2012smart}.

Recent literature has increasingly emphasized the market impact of renewable energy entry \citep{green2015costs}. Meanwhile, \citet{aid2020entry} employed a mean-field game approach to investigate the entry and exit decisions of large populations, albeit with only two types of energy sources considered (conventional and renewable).

\paragraph{Mean-field games and their applications to economic and financial modeling}
Mean-field games (MFG) can be seen as limits of stochastic differential games with identical agents and symmetric interactions, when the number of agents goes to infinity. MFG were introduced in \cite{lasry2007mean}; an already extensive and quickly growing body of literature focuses on theory of these objects and their applications, in particular to economics and finance. The present paper has a particular focus on MFG of optimal stopping, recently studied by, e.g., \cite{bertucci2018optimal,bouveret2020mean,dumitrescu2021control}. In addition to the paper by \cite{aid2020entry}, already quoted above, other applications of mean-field theory (not necessarily optimal stopping) to energy markets  include, e.g., \cite{carmona2022mean,shrivats2022mean,alasseur2022mean,aid2020mckean}.

In the context of entry and exit dynamics, the MFG approach is particularly useful as it allows us to analyze the behavior of a large number of agents and derive insights into the collective behavior of the population. Entry and exit decisions are often influenced not only by the actions of other agents in the market, but also by randomness and uncertainty about future events, such as electricity prices, fuel costs, carbon prices, or investment costs. The MFG framework facilitates the introduction of a probabilistic approach with a large population compared to optimization models, which are usually limited to one agent. This method allows to model the evolution of the population over time accounting for stochastic factors that affect agents' decisions.
MFGs make it possible to introduce heterogeneity among agents in a more straightforward manner than traditional Cournot equilibrium \citep{chan2015bertrand}, while avoiding the computational complexity of Supply Function Equilibrium, by considering different measure flows for different classes of agents.

Finally, the MFG approach is more tractable than agent-based simulation models as it incorporates an equilibrium concept, whose existence and, to a certain extent, uniqueness can be shown rigorously.  However, it is worth noting that MFG can converge to an ABM when the intertemporal preference rate goes to infinity, that is, the anticipation of the agents vanishes \citep{bertucci2019some}.

\section{The model}
\label{model.sec}
Although we are interested in the mean-field game model in this paper, we start by describing the setting with a finite number of agents to motivate the limit as the number of agents goes to infinity, discussed in the next section. 

\subsection{Power plants}

We assume that in the market there is a large number of electricity producers, using either one of $N$ available conventional technologies (coal, gas, etc.) or one of $\overline N$ available renewable technologies. We use the index $i$ to refer to different technologies and the index $j$ to refer to different producers using the same technology. Conventional technologies are numbered from $i=1$ to $i=N$ and renewable technologies are numbered from $i=N+1$ to $i=N+\overline N$.   We assume that there are $N_i$ producers using technology $i$. As several conventional technologies may use the same fuel, we assume that there are $K$ fuels, whose prices are denoted $P^1_t,\dots,P^K_t$, and that $i$-th technology uses fuel $k(i)$.

The life cycle of a power plant is modelled as follows. The decision to build the power plant is taken at time $\tau_1$. Then, at time $\tau_1 + T^i_0$ the plant starts to generate power. The plant is decommissioned at the latest at date $\tau_1 + T^i_1$, where $T^i_1>T^i_0$, but may be decommissioned earlier, at time $\tau_2\in[\tau_1,\tau_1 + T^i_1]$ if it is no longer profitable to operate it (for example, in the situation of asset stranding due, e.g., to the change in environmental regulations). In our model, the electricity market is simulated over the time interval $[0,T]$, and we assume that at time $t=0$, some conventional generation projects are operational or in the construction stage (in that case, their owner wants to optimally decide when to close down the plant, that is, choose $\tau_2$), while for others the decision to build the plant has not been taken yet, and in that case, the owner wants to optimally choose both the date $\tau_1$ for starting the construction and the date $\tau_2$ when the plant will be closed down. We denote by $\lambda_i$ the maximum capacity function of the plant: for example, the above life cycle description corresponds to $\lambda_i(t) = \mathbf 1_{T^i_0\leq t\leq T^i_1}$. In the proofs of our theorems, for technical reasons, we shall assume that $\lambda_i$ is a smooth function, that is, the maximum capacity increases and decreases progressively over a short time interval. This smoothness requirement affects the maximum capacity but not the actual production: each individual plant can decide at any time to stop production immediately, or to increase the output from zero to any fraction of maximum capacity.

\subsection{Conventional producers}
Each conventional producer has marginal cost function $C^{ij}_t:[0,1]\to
\mathbb R$: 
$C^{ij}_t(\xi)$ is the unit cost of increasing capacity if operating at $\xi$

{We assume 
$$
C^{ij}_t(\xi) = f_i e_{k(i)} P^C_t+f_i P^{k(i)}_t + Z^{ij}_t + c^i(\xi),
$$
where $P^C$ is the emission price, $e_{k(i)}$ is the carbon intensity of one unit of fuel $k(i)$, $P^{k(i)}_t$ is the price of fuel $k(i)$, $f_i$ is the conversion ratio (quantity of fuel needed to produce one unit of electricity with technology $i$),} $Z^{ij}$ is a random process which reflects all other costs of the producer when operating at low capacity, and which follows the dynamics
\begin{align}
d Z^{ij}_t = k^i(\theta^i- Z^{ij}_t) dt + \delta^i \sqrt{Z^{ij}_t} dW^{ij}_t,\quad Z^{ij}_0= z_{ij},\label{convdyn}
\end{align}
and $c^i:\mathbb R_+ \to [0,1]$ is increasing smooth function with
$c^i(0)=0$, which reflects additional capacity-dependent costs.  

For a given electricity price $p$, the producer, who is a price taker in our model, will therefore offer in the market the capacity $\xi^*$ such that $C^{ij}_t(\xi^*) = p$, in other words,
{$$
\xi^* = F_i(p -f_i e_{k(i)} P^C_t-f_i P^{k(i)}_t - Z^{ij}_t  ),
$$
where $F_i$ is the inverse function of $c^i$. We assume that $c^i$ is such that $F_i$ is twice differentiable with bounded derivatives for each $i$. The instantaneous gain of the producer at price level $p$ is then given by
\begin{align}
\int_0^{\xi^*} (p-C^{ij}_t(\xi))d\xi = G_i(p- e_{k(i)} P^C_t+P^{k(i)}_t + Z^{ij}_t ),\label{gain.eq}
\end{align}}
where
$$
G_i(x) = \int_0^x F_i(z) dz, \quad x\geq 0,\qquad G_i(x) = 0, \quad x<0.
$$
To obtain the equality \eqref{gain.eq}, we used integration by parts and change of variables. 

To make the model more realistic, we distinguish peak and off-peak electricity demand and price. The peak demand/price correspond to the electricity consumption on weekdays from 7AM to 8PM, and we introduce the coefficients $c_p = \frac{65}{168}$ and $c_{op} = 1-c_p$.  Peak and off-peak prices will be denoted by $P^p_t$ and $P^{op}_t$ respectively.

The agent aims to maximize the quantity{
\begin{multline}
\mathbb E\Big[\int_{\tau_1}^{\tau_2} e^{-\rho
  t}\lambda_i(t-\tau_1)(c_p G_i(P^p_t -f_i  e_{k(i)}P^C_t-f_i P^{k(i)}_t - Z^{ij}_t )\\ +c_{op} G_i(P^{op}_t - f_i e_{k(i)} P^C_t-f_i P^{k(i)}_t - Z^{ij}_t )-\kappa_i) dt \\ - K_i e^{-(\rho+\gamma_i)\tau_1} +
  \widetilde K_i e^{-(\rho+\gamma_i) \tau_2 } \Big]\label{optconv}
\end{multline}}
If the plant is not yet operational/under construction, the optimization is carried out over 
over all $\tau_1,\tau_2\in \mathcal T([0,T])$ such that $\tau_1\leq \tau_2$, otherwise the optimization is only over $\tau_2\in \mathcal T([0,T])$.
Here, $\kappa_i$ is the fixed cost per
unit of time that the producer pays until exiting the market, $K_i$ is the capital cost of building the plant with
technology $i$, $\gamma_i$ is the rate of decay of capital costs and $\widetilde K_i$ is the value recovered when the plant is closed down (scrap value).  
All costs are normalized to one unit of capacity.

\subsection{Renewable producers}

Each renewable power plant in operation generates $S^{ij}_t\in (0,1)$ units of electricity per unit
  time at zero cost, where 
\begin{align}
dS^{ij}_t = k^i(\theta^i- S^{ij}_t) dt + \delta^i \sqrt{S^{ij}_t(1- S^{ij}_t)}
d\overline W^{ij}_t,\quad S^{ij}_0
= s_{ij}\in (0,1).\label{rendyn}
\end{align}

The renewable producers always bid their
  full intermittent capacity. The owners of potential renewable project which have not yet entered the market therefore optimize the quantity
\begin{multline}
 \mathbb E\Big[\int_{\tau_1}^{\tau_2} e^{-\rho t}\lambda_i(t-\tau_1)(
  (c_p P^p_t+c_{op}P^{op}_t) S^i_t - {\kappa_i}) dt \\-{K}_i
e^{-(\rho+\gamma_i) \tau_1} + \widetilde{K}_i  e^{-(\rho+\gamma_i)\tau_2}\Big]\label{optren}
\end{multline}
If the plant is not yet operational/under construction, the optimization is carried out over 
over all $\tau_1,\tau_2\in \mathcal T([0,T])$ such that $\tau_1\leq \tau_2$, otherwise the optimization is only over $\tau_2\in \mathcal T([0,T])$. Although in the model it is possible to decommission a renewable plant before the end of its lifetime, in practice with realistic parameter value this will never be optimal. 

\subsection{Price formation}

As mentioned above, we distinguish peak and off-peak electricity demand.  The peak and off-peak demand will be denoted by $D^p_t$ and $D^{op}_t$. Both are normalized to one hour, so that the total demand is given by $D_t = c_p D^p_t + c_{op} D^{op}_t$. 

The consumption of fuel of type $k$  at time $t$ when the electricity price equals $P^E$ and the fuel price equals $P^k$ writes
$$
\Psi^k_t (P^E, P^k)=\sum_{i: k(i)=k}\sum_{j=1}^{N_i} \lambda_i(t-\tau_1^{ij}) \mathbf 1_{\tau_2^{ij}>t} f_i Q_{ij} F_i(P^E - f_i e_{k(i)} P^C - f_i P^k - Z^{ij}_t), 
$$
where $Q_{ij}$ is the capacity of plant $j$ using technology $i$, while the total electricity generation using fuel $k$ is 
$$
 F^k_t (P^E, P^k)=\sum_{i: k(i)=k}\sum_{j=1}^{N_i} \lambda_i(t-\tau_1^{ij}) \mathbf 1_{\tau_2^{ij}>t}  Q_{ij} F_i(P^E - f_i e_{k(i)} P^C - f_i P^k - Z^{ij}_t).
$$

We assume that the fuel price is determined by matching the consumption with an exogenous supply function:
\begin{align}
c_p \Psi^k_t(P^p_t,P^k_t) + c_{op} \Psi^k_t(P^{op}_t,P^k_t) = \Phi_k(P^k_t).  \label{fprice}
\end{align}
We denote by $R_t$ the total supply of electricity from renewable sources:
$$
R_t = \sum_{i=N+1}^{N+\overline N} \frac{1}{ N_i}\sum_{j=1}^{ N_i}  Q_{ij}\lambda(t-\tau^{ij}_1)  S^{ij}_t,
$$
where $ Q_{ij}$ for $i=N+1,\dots,\overline N$ is the installed capacity of plant $j$ using renewable technology $i$. 

The peak and off-peak electricity prices satisfy the following equations:
\begin{align}
&(D^p_t  - R_t)^+ = F_0(P^p_t) + \sum_{k=1}^K F^k_t(P^p_t, P^k_t),\label{eprice}\\
&\text{or} \quad (D^p_t-R_t)^+  >  F_0(P^p_t) + \sum_{k=1}^K F^k_t(P^p_t, P^k_t)\quad \text{and} \quad P^p_t = P^*.\notag
\end{align}
and similarly for $P^{op}_t$. Here $F_0$ is the exogenous \emph{baseline supply} function, which is increasing and satisfies $F_0(0) = 0$, and $P^*$ is the price cap in the market. The second inequality corresponds to the so called loss of load situation, when the price cap is reached and the demand is not entirely satisfied by the production. 

To summarize, in our model, the CO2 emissions price is exogenous, and the the prices of fuels and electricity are determined endogenously, by solving equations \eqref{fprice} and \eqref{eprice}. 

\section{Entry and exit mean-field game}
\label{mfg.sec}
\subsection{Definition of the mean-field setting}
To formalize and solve the model of the previous section in a mean-field game context, we consider a mean-field game of optimal stopping with several classes of agents, where agents can change class, that is, when agents in one class stop, they move to another class. {\color{black}For a given technology $i$, the agents will therefore belong to one of two classes: the class of plants for which the decision to build has not been taken yet (denoted by $\widehat C_i$), and the class of plants, which are operational or under construction (denoted by $C_i$). Upon making the decision to build the plant, the agent in the class $\widehat C_i$ moves to the class $C_i$. 

Having in mind our specific application, and taking into account the fact that the age variable only becomes relevant once the decision to build the plant has been made, we assume that the state of an agent of class $\widehat C_i$ belongs to $\mathcal O_i$, and that the state of an agent of class $C_i$ belongs to $\mathcal A\times \mathcal O_i$, where $\mathcal A$ is a compact subset of $\mathbb R$ and $\mathcal O_i$ is an open subset of $\mathbb R$ (the first component will model the age of the plant, and the second component will correspond to the state process of the plant, that is, the cost process for conventional plants and the capacity factor process for renewable plants)}. We denote by $\mathcal V_i$ the space of positive measure flows on $\mathcal A\times \overline{\mathcal O}_i$ such that for $(m_t)_{0\leq t\leq T}$, 
$$
\int_0^T m_t(\mathcal A\times \overline{\mathcal O}_i) dt <\infty,
$$
and by $\mathcal M_i$ the space of finite positive measures on $[0,T]\times \mathcal A \times \overline {\mathcal O}_i$. Similarly, we denote by $\widehat{\mathcal V}_i$ the space of positive measure flows on $ \overline{\mathcal O}_i$ such that for $(m_t)_{0\leq t\leq T}$, 
$$
\int_0^T m_t(  \overline{\mathcal O}_i) dt <\infty,
$$
and by $\widehat{\mathcal M}_i$ the space of finite positive measures on $[0,T]\times \overline {\mathcal O}_i$. 

\textcolor{black}{When the number of producers is finite, the class $C_i$ is characterized by a triple $(\nu^{i,N_i},\mu^{i,N_i},(m^{i,N_i}_t)_{0\leq t\leq T})$, where $\nu^{i,N_i}$ and $\mu^{i,N_i}$ are random measures taking values in $\mathcal M_i$ and $(m^{i,N_i}_t)_{0\leq t\leq T}$ is a flow of random measures taking values in $\mathcal V_i$. The measure $\nu^{i,N_i}$ will be referred to as the \emph{entry measure}, the measure $\mu^{i,N_i}$ as the \emph{exit measure}, and the measure flow $m^{i,N_i}_t$ as the \emph{occupation measure flow}. Similarly, the class $\widehat C_i$ is characterized by a triple $(\hat\nu^{i,N_i},\hat\mu^{i,N_i},(\hat m^{i,N_i}_t)_{0\leq t\leq T})$, where $\hat\nu^{i,N_i}$ and $\hat\mu^{i,N_i}$ are random measures taking values in $\widehat{\mathcal M}_i$ and $(\hat m^{i,N_i}_t)_{0\leq t\leq T}$ is a flow of random measures taking values in $\widehat{\mathcal V}_i$.}

\textcolor{black}{Recall that, for a given technology $i$,  the $j$-th conventional producer is operational or under construction between $\tau_1^{ij}$ and $\tau_2^{ij}$ and has capacity $Q_{ij}$. 
The random measures and measure flows characterizing classes $C_i$ and $\widehat C_i$, for $i=1,\ldots, N$ are defined as follows:
\begin{align*}
m_t^{i,N_{i}}(da,dx)&=\sum_{j=1}^{N_i} Q_{ij} \delta_{(t-\tau^{ij}_1,Z_t^{ij})}(da,dx) \textbf{1}_{\tau_1^{ij}<t
\leq\tau_2^{ij}},\\
\mu^{i,N_i}(dt,da,dx)&= \sum_{j=1}^{N_i} Q_{ij} \delta_{(\tau_2^{ij},\tau_2^{ij}-\tau^{ij}_1,Z_{\tau_2^{ij}})}(dt,da,dx),\\
\nu^{i,N_i}(dt,da,dx)&= \sum_{j=1}^{N_i} Q_{ij} \delta_{(\tau_1^{ij},Z_{\tau_1^{ij}})}(dt,dx)\delta_0(da)+\sum_{j=1}^{N_i} Q_{ij} \delta_{(-\tau^{ij}_1,Z_{0}^{ij})}(da,dx)\delta_0(dt),\\
\hat{\mu}^{i,N_i}(dt,dx)&= \sum_{j=1}^{N_i} Q_{ij} \delta_{(\tau_1^{ij},Z_{\tau_1^{ij}})}(dt,dx),\\
\hat{m}_t^{i,N_{i}}(dx)&=\sum_{j=1}^{N_i} Q_{ij} \delta_{(Z_t^{ij})}(dx) \textbf{1}_{t\leq\tau_1^{ij}},\\
\hat{\nu}^{i,N_i}(dt,dx)&= \sum_{j=1}^{N_i} Q_{ij} \mathbf 1_{\tau^{ij}_1>0}\delta_{Z_{0}^{ij}}(dx)\delta_{0}(dt).\end{align*}
The random measures and measure flows characterizing classes $C_i$ and $\widehat C_i$, for $i=N+1,\ldots, N+\bar{N}$ are defined as above, by replacing the process $Z^{ij}$ by $S^{ij}$.}

{\color{black}In the mean-field game limit, we assume that the number of agents in each class, $N_i$, tends to infinity, and the capacity of each agent, $Q_{ij}$, tends to zero, such that, for each technology $i$, one can define a non-degenerate limiting distribution of states of producers who are yet to enter the market $\hat\nu_0^i$ and a limiting distribution of ages and states of producers who are under construction or operational $\nu^i_0$:
\begin{align}
\sum_{j=1}^{N_i} Q_{ij} \mathbf 1_{\tau^{ij}_1>0}\delta_{(Z^{ij}_0)}\xrightarrow[N_i \to \infty]{d} \hat \nu^i_0,\qquad \sum_{j=1}^{N_i} Q_{ij} \delta_{(-\tau^{ij}_1,Z^{ij}_0)}\xrightarrow[N_i \to \infty]{d}  \nu^i_0.\label{limQ}
\end{align}
}

{\color{black} In the mean-field limit, we define the state process of the representative conventional agent
$$
d Z^{i}_t = k^i(\theta^i- Z^{i}_t) dt + \delta^i \sqrt{Z^{i}_t} dW^{i}_t,
$$
and introduce the deterministic counterparts
$(\nu^{i},\mu^{i},(m^{i}_t)_{0\leq t\leq T})\in \mathcal M_i \times \mathcal M_i \times \mathcal V_i$ and $(\hat\nu^{i},\hat\mu^{i},(\hat m^{i}_t)_{0\leq t\leq T})\in \widehat{\mathcal M}_i \times \widehat{\mathcal M}_i \times \widehat{\mathcal V}_i$, of the random triples defined above:
\begin{align}
m_t^{i}(da,dx)&=\int_{\mathcal A \times \overline{\mathcal O}_i}\nu^i_0(da',dx') \mathbb E^{(x')}[\delta_{(a'+t,Z_t^{i})}(da,dx) \textbf{1}_{t
\leq\tau_2^{i}}]\notag\\&\qquad \qquad+\int_{\overline{\mathcal O}_i}\hat\nu^i_0(dx') \mathbb E^{(x')}[\delta_{(t-\tau^i_1,Z_t^{i})}(da,dx) \textbf{1}_{\tau^i_1< t
\leq\tau_2^{i}}],\label{ocbegin}\\
\mu^{i}(dt,da,dx)&= \int_{\mathcal A \times \overline{\mathcal O}_i}\nu^i_0(da',dx')\mathbb E^{(x')}[\delta_{(\tau_2^{i},\tau_2^{i}+a',Z_{\tau_2^{i}})}(dt,da,dx)]\notag\\&\qquad\qquad+\int_{ \overline{\mathcal O}_i}\hat\nu^i_0(dx')\mathbb E^{(x')}[\delta_{(\tau_2^{i},\tau_2^{i}-\tau_1^i,Z_{\tau_2^{i}})}(dt,da,dx)],\\
\nu^{i}(dt,da,dx)&= \int_{ \overline{\mathcal O}_i}\hat\nu^i_0(dx')\mathbb E^{(x')}[\delta_{(\tau_1^{i},Z_{\tau_1^{i}})}(dt,dx)]\delta_0(da)+\nu^i_0 (da,dx)\delta_0(dt),\\
\hat{\mu}^{i}(dt,dx)&= \int_{\overline{\mathcal O}_i} \hat\nu^i_0(dx')\mathbb E^{(x')}[\delta_{(\tau_1^{i},Z_{\tau_1^{i}})}(dt,dx)].\\
\hat{m}_t^{i}(dx)&= \int_{\overline{\mathcal O}_i} \hat\nu^i_0(dx')\mathbb E^{(x')}[ \delta_{(Z_t^{i})}(dx) \textbf{1}_{t\leq\tau_1^{i}}].\\
\hat{\nu}^{i}(dt,dx)&= \hat\nu^i_0(dx)\delta_{0}(dt).\label{ocend}
\end{align}
for $i=1,\dots,N$, here
 $\mathbb E^{(x')}$ denotes the expectation under the probability measure where $Z^i_0 = x'$. The triples for $i=N+1,\dots, N+\overline N$ are defined similarly, replacing the cost process for conventional agent $Z^i$ by the capacity factor process for renewable agents $S^i$.}

These deterministic measures and measure flows satisfy a linear constraint, which we now describe. 

Let $\widehat{\mathcal L}_i$ be the infinitesimal generator of the state process of the representative agent and $\mathcal L_i$ be the infinitesimal generator extended to the age variable, that is, for a function $u\in C^{1,1,2}_b([0,T]\times \mathcal A \times \overline{\mathcal O}_i)$, 
$$
\mathcal L_i u = \widehat{\mathcal L}_i u + \frac{\partial u}{\partial a}.
$$
{\color{black}For the conventional producers the infinitesimal generator of the state (cost) process takes the form
$$
\widehat{\mathcal L}_i u = k^i(\theta^i-x) \frac{\partial u}{\partial x} + \frac{(\delta^i)^2 x}{2} \frac{\partial^2 u}{\partial x^2},\quad i=1,\dots, N,
$$
while for the renewable producers, the state (capacity factor) process takes the form 
$$
\widehat{\mathcal L}_i u = k^i(\theta^i-x) \frac{\partial u}{\partial x} + \frac{(\delta^i)^2 x(1-x)}{2} \frac{\partial^2 u}{\partial x^2},\quad i= N+1,\dots,N+\overline N.
$$}

\textcolor{black}{For any stopping times $\tau_1^{i}, \tau_2^{i} \in \mathcal{T}([0,T])$ and test function $u \in C^{1,1,2}_b([0,T]\times \mathcal A \times \overline{\mathcal O}_i)$, by applying Itô's formula between $\tau_1^{i}$ and $\tau_2^{i}$ to $u(t,t-\tau^i_1,Z^{i}_t)$, we obtain 
\begin{multline*}u(\tau_2^{i},\tau_2^{i}-\tau^i_1, Z^i_{\tau_2^{i}}) = 
u(\tau_1^{i},0, Z^i_{\tau_1^{i}})\\+\int_{\tau_1^{i}}^{\tau_2^{i}}\left\{\frac{\partial u}{\partial t} + \mathcal L_i u\right\}(t,t-\tau^i_1,Z^i_t)\, dt+\int_{\tau_1^{i}}^{\tau_2^{i}}\delta^{i}\sqrt{Z_t^{i}}\partial_xu(t,t-\tau^i_1, Z_t^{i})dW_t^{i},
\end{multline*}
whereas by applying the Itô formula between $0$ and $\tau^i_2$ to $u(t,t+a',Z^i_t)$ (for some constant $a'\geq0$), we get:
\begin{multline*}u(\tau_2^{i},\tau_2^{i}+a', Z^i_{\tau_2^{i}}) = 
u(0,a', Z^i_0)\\+\int_{0}^{\tau_2^{i}}\left\{\frac{\partial u}{\partial t} + \mathcal L_i u\right\}(t,t+a',Z^i_t)\, dt+\int_{0}^{\tau_2^{i}}\delta^{i}\sqrt{Z_t^{i}}\partial_xu(t,t+a', Z_t^{i})dW_t^{i}.
\end{multline*}
By taking the expectation of both equations, assuming that $Z^i_0 = x'$,  we get:
\begin{align*}
\mathbb E^{(x')}[u(\tau_2^{i},\tau_2^{i}-\tau^i_1, Z^i_{\tau_2^{i}})] &= 
\mathbb E^{(x')}[u(\tau_1^{i},0, Z^i_{\tau_1^{i}})]\\&+\int_0^T \mathbb E^{(x')}\left[\left\{\frac{\partial u}{\partial t} + \mathcal L_i u\right\}(t,t-\tau^i_1,Z^i_t)\mathbf 1_{\tau^i_1<t\leq \tau^i_2}\right] dt\\
\mathbb E^{(x')}[u(\tau_2^{i},\tau_2^{i}+a', Z^i_{\tau_2^{i}})] &= 
u(0,a',x')\\&+\int_0^T \mathbb E^{(x')}\left[\left\{\frac{\partial u}{\partial t} + \mathcal L_i u\right\}(t,t+a',Z^i_t)\mathbf 1_{t\leq \tau^i_2}\right] dt,\end{align*}
which can be equivalently written as
\begin{align*}
&\int_{[0,T]\times \mathcal A \times \overline{\mathcal  O}_i} u(t,a,x) \mathbb E^{(x')}[\delta_{(\tau_2^{i},\tau_2^{i}-\tau^i_1, Z^i_{\tau_2^{i}})}(dt,da,dx)] \\ & = \int_{[0,T]\times \mathcal A \times \overline{\mathcal  O}_i} u(t,a,x) \mathbb E^{(x')}[\delta_{(\tau^i_1, Z^i_{\tau_1^{i}})}(dt,dx)]\delta_0(da)\\
 &\qquad \qquad+ \int_{[0,T]\times \mathcal A \times \overline{\mathcal  O}_i}\left\{\frac{\partial u}{\partial t} + \mathcal L_i u\right\}\mathbb E^{(x')}[\delta_{(t-\tau^i_1,Z^i_t)}(da,dx)\mathbf 1_{\tau^i_1<t\leq \tau^i_2}]\\
 &\int_{[0,T]\times \mathcal A \times \overline{\mathcal  O}_i} u(t,a,x) \mathbb E^{(x')}[\delta_{(\tau^i_2,\tau^i_2 + a',Z_{\tau^i_2})}(dt,da,dx)] \\ &= \int_{[0,T]\times \mathcal A \times \overline{\mathcal  O}_i} u(t,a,x) \delta_{(0,a',x')}(dt,da,dx)\\
 &\qquad \qquad  + \int_{[0,T]\times \mathcal A \times \overline{\mathcal  O}_i}\left\{\frac{\partial u}{\partial t} + \mathcal L_i u\right\}\mathbb E^{(x')}[\delta_{(t+a,Z^i_t)}(da,dx)\mathbf 1_{t\leq \tau^i_2}]. 
\end{align*}
Integrating the first equation with respect to $\hat \nu^i_0(dx')$, the second one with respect to $\nu^i_0(da',dx')$ adding them together, 
 and using the definition of the measures $(\nu^{i}, \mu^{i}, (m_t^{i})_{0 \leq t \leq T})$, we get that these measures satisfy the linear programming constraint }
\begin{multline}\label{eq}
\int_{[0,T]\times \mathcal A \times \overline{\mathcal  O}_i} u(t,a,x) \nu^i(dt,da,dx) + \int_{[0,T]\times \mathcal A \times \overline{\mathcal  O}_i} \left\{\frac{\partial u}{\partial t} + \mathcal L_i u\right\}m^i_t(da,dx)\, dt \\= \int_{[0,T]\times \mathcal A \times \overline{\mathcal  O}_i} u(t,a,x) \mu^i(dt,da,dx)
\end{multline}
for all $u\in C^{1,1,2}_b([0,T]\times \mathcal A \times \overline{\mathcal O}_i)$. A similar argument shows that the triple $(\hat\nu^i,\hat\mu^i,(\hat m^i_t)_{0\leq t\leq T})$ satisfies the constraint 
\begin{multline}\label{eqhat}
\int_{[0,T]\times \overline{\mathcal  O}_i} u(t,x) \hat\nu^i(dt,dx) + \int_{[0,T]\times \overline{\mathcal  O}_i} \left\{\frac{\partial u}{\partial t} + \widehat{\mathcal L}_i u\right\}\hat m^i_t(dx)\, dt \\= \int_{[0,T] \times \overline{\mathcal  O}_i} u(t,x) \hat\mu^i(dt,dx)
\end{multline}
for all $u\in C^{1,2}_b([0,T]\times \overline{\mathcal O}_i)$.

Since moves are possible only from class $\widehat C_i$ to class $C_i$, the entry measures satisfy the following constraints:
\begin{align}
\hat\nu^i(dt,dx) &= \hat \nu^{i}_0(dx)\, \delta_0(dt),\label{constnuhat}\\ \nu^i(dt,da,dx) &= \nu^{i}_0(da,dx)\, \delta_0(dt) + \hat \mu^i(dt,dx)\,\delta_0(da). \label{constnu}
\end{align}
Recall that here $\hat \nu^{i}_0(dx)$ denotes the initial distribution of states of agents in class $\widehat C_i$, and $\nu^{i}_0(da,dx)$ denotes the initial distribution of ages and states of agents in class $C_i$. 
In other words, the entry measure of the class $\widehat C_i$ is only given by the initial condition, while the entry measure of the class $C_i$ is given by the initial condition plus the exit measure of the class $\widehat C_i$.

{\color{black}The relaxed linear programming approach to mean-field games, developed in \citep{bouveret2020mean,dumitrescu2021control} and described in the rest of this section in the specific context of the model of this paper, consists in looking for Nash equilibrium over all measures satisfying the linear constraints (\ref{eq}--\ref{constnu}) rather than just the occupation measures defined in (\ref{ocbegin}--\ref{ocend}). }


 \subsection{Price formation}
Using the notation of this section, in the limiting setting of mean-field games, the electricity supply from conventional
agents using fuel $k$ at time $t$, with electricity price level $P^E$ and fuel price $P^k$ is given by
$$
F^k_t(P^E, P^k) = 
 \sum_{i:k(i)=k}\int_{\mathcal A\times  \overline{\mathcal O}_i} m^i_t(da,dx)
\lambda_i (a) F_i(P^E -f_i  e_k P^C - f_i P^k - x),
$$
while the consumption of fuel $k$ writes
$$
\Psi^k_t(P^E,P^k)=\sum_{i:k(i)=k}\int_{\mathcal A\times  \overline{\mathcal O}_i} m^i_t(da,dx)
\lambda_i (a) f_i F_i(P^E -f_i  e_k P^C - f_i P^k - x)
$$
and the supply from renewable producers is given by
$$
R_t = \sum_{i=N+1}^{N+\overline N} \int_{\mathcal A \times  \overline{\mathcal O}_i} m^i_t(da,dx)
x\,\lambda_i (a),
$$

The electricity and fuel prices are then determined by solving the equations \eqref{fprice}-\eqref{eprice}. 

The following proposition establishes the existence and uniqueness of prices obtained with this procedure, and provides a convenient way to compute them. 
\begin{proposition}\label{price.prop}
Let $\overline \Phi_k(x) := \int_0^x \Phi_k(y) dy$ for $k=1,\dots,K$ and $G_0(x):=\int_0^x F_0(y) dy$. Assume that the functions $\overline \Phi_1,\dots,\overline \Phi_K$ and $G_0$ are strongly convex. Then, for each $t\in [0,T]$, there exists a unique solution of the system \eqref{fprice}-\eqref{eprice}. 
\end{proposition}
\begin{proof}
Define the functions
$$
G^k_t(x,y):= \sum_{i:k(i)=k}\int_{\mathcal A\times  \overline{\mathcal O}_i} m^i_t(da,dz)
\lambda_i (a) G_i(x - f_i e_k P^C  - f_i y- z).
$$
for $k=1,\dots,K$, and
\begin{align*}
&G_t(P^p, P^{op}, P^1,\dots,P^K):= \sum_{k=1}^K \{c_p G^k_t(P^p,P^k)+c_{op}  G^k_t(P^{op},P^k)\}+c_p G_0(P^p) \\ & \qquad  + c_{op}G_0(P^{op})- c_p P^p (D^p_t-R_t)^+ - c_{op} P^{op} (D^{op}_t-R_t)^+ + \sum_{k=1}^K \overline \Phi_k(P^k) 
\end{align*}
It is easy to see that the solution the system \eqref{fprice}-\eqref{eprice} corresponds to the minimizer of $G_t$ on the set $[0,P^*]^2 \times \mathbb R^K_+$. Existence and uniqueness then follow from the assumptions of the proposition.  
\end{proof}
\subsection{Optimization functionals}
From \eqref{optconv}, the optimization functional for conventional agents using technology $i$ is given by
\begin{multline}
\int_{[0,T]\times \mathcal A \times \overline{\mathcal O}_i} m^i_t(da, dx) e^{-\rho
  t}\lambda_i(a)(c_p G_i(P^p_t - f_i e_{k(i)} P^C_t-f_i P^{k(i)}_t - x )\\ +c_{op} G_i(P^{op}_t - f_i e_{k(i)} P^C_t-f_i P^{k(i)}_t - x )-\kappa_i) dt \\ - K_i \int_{[0,T] \times \overline{\mathcal O}_i} \hat \mu^i(dt,  dx) e^{-(\rho+\gamma_i)
  t}
  +
  \widetilde K_i \int_{[0,T]\times \mathcal A \times \overline{\mathcal O}_i}  \mu^i(dt, da, dx) e^{-(\rho+\gamma_i) t},\label{optconv.mfg}
\end{multline}
for $i=1,\dots,N$,
while the optimization functional for renewable agents
using technology $i$
is given by
\begin{multline}
\int_{[0,T]\times \mathcal A \times \overline{\mathcal O}_i} m^i_t(da, dx) e^{-\rho
  t}\lambda_i(a)(
  (c_p P^p_t+c_{op}P^{op}_t) x - {\kappa_i}) dt \\- K_i \int_{[0,T] \times \overline{\mathcal O}_i} \hat \mu^i(dt, dx) e^{-(\rho+\gamma_i)
  t}
  +
  \widetilde K_i \int_{[0,T]\times \mathcal A \times \overline{\mathcal O}_i}  \mu^i(dt, da, dx) e^{-(\rho+\gamma_i) t},\label{optren.mfg}
\end{multline}
for $i=N+1,\dots,N+\overline N$. 

\subsection{Definition and properties of Nash equilibrium}
For $i=1,\dots,N+\overline N$, let $\hat \nu^{i}_0(dx)$ be the initial distribution of states of electricity generation projects for which the decision to build has not been taken yet, and let $\nu^{i}_0(da\times dx)$ be the initial distribution of ages and states of plants, which are operational or under construction, the age being counted from the build decision. We define by $\mathcal R_i(\hat\nu^i_0,\nu^i_0)$ the class of $4$-uplets 
$$
(\hat \mu^i, (\hat m^i_t)_{0\leq t\leq T},\mu^i, ( m^i_t)_{0\leq t\leq T})\in  \widehat{\mathcal M}_i\times \widehat{\mathcal V}_i\times \mathcal M_i \times \mathcal V_i,
$$
satisfying the following constraints: 
\begin{align}
\int_{[0,T]\times \mathcal A \times \overline{\mathcal  O}_i} u(t,a,x) \nu^i(dt,da,dx) &+ \int_{[0,T]\times \mathcal A \times \overline{\mathcal  O}_i} \left\{\frac{\partial u}{\partial t} + \mathcal L_i u\right\}m^i_t(da,dx)\, dt \notag\\&= \int_{[0,T]\times \mathcal A \times \overline{\mathcal  O}_i} u(t,a,x) \mu^i(dt,da,dx)\label{constnash1}\\
\int_{[0,T]\times \overline{\mathcal  O}_i} \hat u(t,x) \hat\nu^i(dt,dx) &+ \int_{[0,T]\times  \overline{\mathcal  O}_i} \left\{\frac{\partial \hat u}{\partial t} + \widehat{\mathcal L}_i \hat u\right\}\hat m^i_t(dx)\, dt \notag\\&= \int_{[0,T]\times \overline{\mathcal  O}_i} \hat u(t,x) \hat\mu^i(dt,dx)\label{constnash2}
\end{align}
for all $u\in C^{1,2,2}_b([0,T]\times \mathcal A \times \overline{\mathcal O}_i)$ and $\hat u\in C^{1,2}_b([0,T] \times \overline{\mathcal O}_i)$, where the entry measures are given by
\begin{align*}
&\hat\nu^i(dt,dx) = \delta_0(dt)\times  \hat \nu^{i}_0(dx)\, ,\\ &\nu^i(dt,da,dx) = \nu^{i}_0(da,dx)\, \delta_0(dt) + \hat \mu^i(dt,dx)\, \delta_0(da).
\end{align*}
When this does not create confusion, we will sometimes drop the arguments and denote this class simply by $\mathcal R_i$. 
\begin{definition}
A relaxed mean-field Nash equilibrium is a sequence of $4$-uplets of measures/measure flows 
$$
(\hat \mu^i, (\hat m^i_t)_{0\leq t\leq T},\mu^i, ( m^i_t)_{0\leq t\leq T}) \in \mathcal R_i,\quad i=1,\dots,N+\overline N,
$$
and the price functions
$$
P^p,P^{op}, P^1,\dots, P^K : [0,T]\to [0,P^*]^2\times \mathbb R^K_+,
$$
such that 
\begin{itemize}
    \item[i.] For each $i=1,\dots,N$, the measures $(\hat \mu^i, (\hat m^i_t)_{0\leq t\leq T},\mu^i, ( m^i_t)_{0\leq t\leq T})$ maximize the functional \eqref{optconv.mfg} over all elements of $\mathcal R_i$; 
    \item[ii.] For each $i=N+1,\dots,N+\overline N$, the measures $(\hat \mu^i, (\hat m^i_t)_{0\leq t\leq T}, \mu^i, ( m^i_t)_{0\leq t\leq T})$ maximize the functional \eqref{optren.mfg} over all elements of $\mathcal R_i$; 
    \item[iii.] For each $t$, the price vector $(P^p_t,P^{op}_t, P^1_t,\dots, P^K_t)$ is uniquely given by Proposition \ref{price.prop}. 
\end{itemize}
\end{definition}

\begin{theorem}\label{main.thm}
Assume that the initial measures satisfy
\begin{align}
\int_{{\mathcal O}_i} \ln(1+|x|)\hat \nu^i_0(dx) + \int_{\mathcal A\times {\mathcal O}_i} \ln(1+|x|)\nu^i_0(da,dx)<\infty. \label{integr.eq}
\end{align}
for $i=1,\dots,N+\overline N$. 
Assume that the peak demand $D^p$, the off-peak demand $D^{op}$ and the carbon price $P^C$ have finite variation on $[0,T]$, and that the conditions of Proposition \ref{price.prop} are satisfied. Then there exists a relaxed mean-field Nash equilibrium.
\end{theorem}
\begin{proposition}\label{unique.prop}
 Let $(P^p_t,P^{op}_t, P^1_t,\dots, P^K_t)$ and $(\overline P^p_t,\overline P^{op}_t, \overline P^1_t,\dots, \overline P^K_t)$ be price vectors corresponding to two Nash equilibria. Then, outside a set of measure zero, 
    $$
    P^p_t = \overline P^p_t,\ P^{op}_t = \overline P^{op}_t,\ P^1_t = \overline P^1_t,\dots, P^K_t = \overline P^K_t.
    $$
\end{proposition}
\section{Numerical resolution and illustrations}
\label{numerics.sec}
We implement a linear programming fictitious play algorithm (see \cite{aid2020entry} and \cite{dumitrescu2022linear}), which consists in the following steps:
\begin{itemize}
\item Choose initial measures and measure flows 
$$
(\hat \mu^{i,0}, (\hat m^{i,0}_t)_{0\leq t\leq T}, \mu^{i,0}, ( m^{i,0}_t)_{0\leq t\leq T})\in \mathcal R_i,\quad i=1,\dots,N+\overline N.
$$ 
\item For $j=1,\dots,N_{iter}$, repeat:
\begin{itemize}
    \item Compute the prices $(P^p_t,P^{op}_t, P^1_t,\dots, P^K_t)_{0\leq t\leq T}$, using the measures and measure flows of iteration $j-1$; 
    \item Compute the best responses 
    $$
    (\hat{\bar\mu}^{i,j}, (\hat{\bar m}^{i,j}_t)_{0\leq t\leq T}, \bar\mu^{i,j}, ( \bar m^{i,j}_t)_{0\leq t\leq T})
    $$
    by minimizing the functionals \eqref{optconv.mfg} and \eqref{optren.mfg} over the set $\mathcal R_i$. 
    \item Update the measures:
\begin{align*}
&(\hat{\mu}^{i,j}, (\hat{ m}^{i,j}_t)_{0\leq t\leq T}, \mu^{i,j}, ( m^{i,j}_t)_{0\leq t\leq T})\\ &= \varepsilon_j (\hat{\bar\mu}^{i,j}, (\hat{\bar m}^{i,j}_t)_{0\leq t\leq T}, \bar\mu^{i,j}, ( \bar m^{i,j}_t)_{0\leq t\leq T})\\ &+ (1-\varepsilon_j) (\hat{\mu}^{i,j-1}, (\hat{ m}^{i,j-1}_t)_{0\leq t\leq T}, \mu^{i,j-1}, (  m^{i,j-1}_t)_{0\leq t\leq T})
\end{align*}
\end{itemize}
\end{itemize}
The best responses are computed by solving a linear programming problem, after discretizing the measures/measure flows, the constraints and the optimization objectives. 

Assuming that $m$ and $\hat m$ have densities, and that $\mu$ and $\hat \mu$ have densities on $[0,T)$, by formal integration by parts, we get: 
\begin{align*}
&\int_{ \overline{\mathcal  O}_i} \hat u(0,x) (\hat\nu^i_0(dx)-\hat m_0^i(dx)) + \int_{[0,T] \times \overline{\mathcal  O}_i} \hat u(t,x)\left\{-\frac{\partial  \hat m^i_t}{\partial t} + \mathcal L^{0*}_i \hat m^i_t - \hat \mu^i(t,x)\right\}\, dt\, dx  \\ &+\int_{ \overline{\mathcal  O}_i} \hat u(T,x) (\hat m_T^i(dx)-\hat\mu^i(\{T\},dx)) = 0\notag
\end{align*}
 for all $i=1,\dots,N+\overline N$ and $v\in C^{1,2}_b([0,T]\times  \overline{\mathcal O}_i)$ and 
\begin{align*}
&\int_{ \mathcal A \times \overline{\mathcal  O}_i} u(0,a,x) (\nu^i_0(da,dx)-m^i_0(da,dx))+\int_{[0,T] \times \overline{\mathcal  O}_i} u(t,0,x) (\hat \mu^i(dt,dx) -m^i_t(dx) \, dt)\\&\qquad + \int_{[0,T]\times \mathcal A \times \overline{\mathcal  O}_i} u(t,a,x)\left\{-\frac{\partial m^i_t}{\partial t} + \mathcal L^*_i m^i_t - \mu^i(t,a,x)\right\}\,dt\,da\,dx \\&\qquad + \int_{ \mathcal A \times \overline{\mathcal  O}_i} u(T,a,x) (m^i_T(da,dx)-\mu^i(\{T\},da,dx)=0
\end{align*}
for all $i=1,\dots,N+\overline N$ and $u\in C^{1,2,2}_b([0,T]\times \mathcal A \times \overline{\mathcal O}_i)$. 
We conclude that $\hat m^i$ and $m^i$ satisfy the following Fokker-Planck equations in the sense of distributions:
\begin{align*}
&-\frac{\partial  \hat m^i_t}{\partial t} + \mathcal L^{0*}_i \hat m^i_t - \hat \mu^i(t,x) = 0,\quad \text{on $(t,x)\in [0,T)\times {\mathcal O}_i$},\quad \hat m^i_0 = \hat \nu^i_0\\
&-\frac{\partial   m^i_t}{\partial t} + \mathcal L^{*}_i  m^i_t -  \mu^i(t,a,x) = 0,\quad \text{on $(t,a,x)\in [0,T)\times \mathcal A\times {\mathcal O}_i$},\quad  m^i_0 =  \nu^i_0,
\end{align*}
with $m^i_t(0,x) = \hat \mu^i(t,x)$, for $i=1,\dots,N+\overline N$.

If the plant age is not a relevant variable for the problem, the latter two equation takes a slightly different form:
\begin{align*}
-\frac{\partial   m^i_t}{\partial t} + \mathcal L^{0*}_i  m^i_t +\hat \mu^i(t,x)-  \mu^i(t,x) = 0,\quad \text{on $(t,x)\in [0,T)\times {\mathcal O}_i$},\quad  m^i_0 =  \nu^i_0,
\end{align*}
for $i=1,\dots,N+\overline N$.

For the computation of the best response, these Fokker-Planck equations are discretized using the implicit scheme, as described in \cite{aid2020entry}, and the optimization functionals \eqref{optconv.mfg} and \eqref{optren.mfg} are computed from the discretized measures in a natural way. Our numerical implementation of the best response computation uses the Gurobi optimization library (\texttt{www.gurobi.com}). 

\paragraph{Numerical illustration} {\color{black}For the numerical illustration we consider a model with nonzero construction time, but infinite lifetime. Given that we perform the simulation over 25 years {\color{black}(starting from 2018)}, this is not a very important limitation. We assume a 2-year build time for renewable plants and 4-year build time for gas plants.\footnote{These numbers correspond to the average 2018 values, see {\texttt{iea.org/data-and-statistics/charts/average-\\power-generation-construction-time-capacity-weighted-2010-2018}}}. }

We consider an electricity market with three types of agents: coal-fired generators, gas-fired generators and renewable (wind) generators. The carbon tax grows linearly from 30 to 60 euros per ton of CO2 for the first 10 years and then from 60 to 200 euros per ton of CO2 during the last 15 years. The code for the implementation is available at \texttt{https://github.com/petertankov/mfg\_elec}; the values of all model parameters may be found in this repository. 

Figures \ref{demand.fig}, \ref{price.fig} and \ref{supply.fig} illustrate the complex dynamics of the electricity industry described by our model. At first, the renewable technology is not yet mature, but the carbon tax is already too high for the coal producers, so they are replaced by gas-fired generators, until about 2025 (see Figure \ref{price.fig}, right graph). Due to higher demand for gas, the price of gas increases, while the price of coal decreases. Then, renewable technology matures and at the same time becomes more competitive because of higher carbon tax. The remaining coal producers and part of the gas producers then leave the market and are replaced with renewable generators, although some gas producers remain in the market until the end to meet the peak electricity demand. The price of electricity in the long term seems to be more influenced by the carbon price than by the gas price. {\color{black}Finally, Figure \ref{timetobuild.fig} illustrates the impact of plant construction time on the installed capacity of different technologies and the electricity price, by comparing the main simulation results with the ones obtained with zero construction time. Nonzero construction time delays the energy transition: in the short term, coal-fired power plant closures are delayed to compensate for later arrival of gas-fired plants, and in the meduim-term, more gas-fired plants are built to compensate for longer construction time of renewables. The impact on electricity prices is less prominent: in the short term, late arrival of gas-fired plants leads to higher peak prices, and in the medium term, off-peak prices are higher due to lower renewable penetration.} 

\begin{figure}
\centerline{\includegraphics[width=\textwidth]{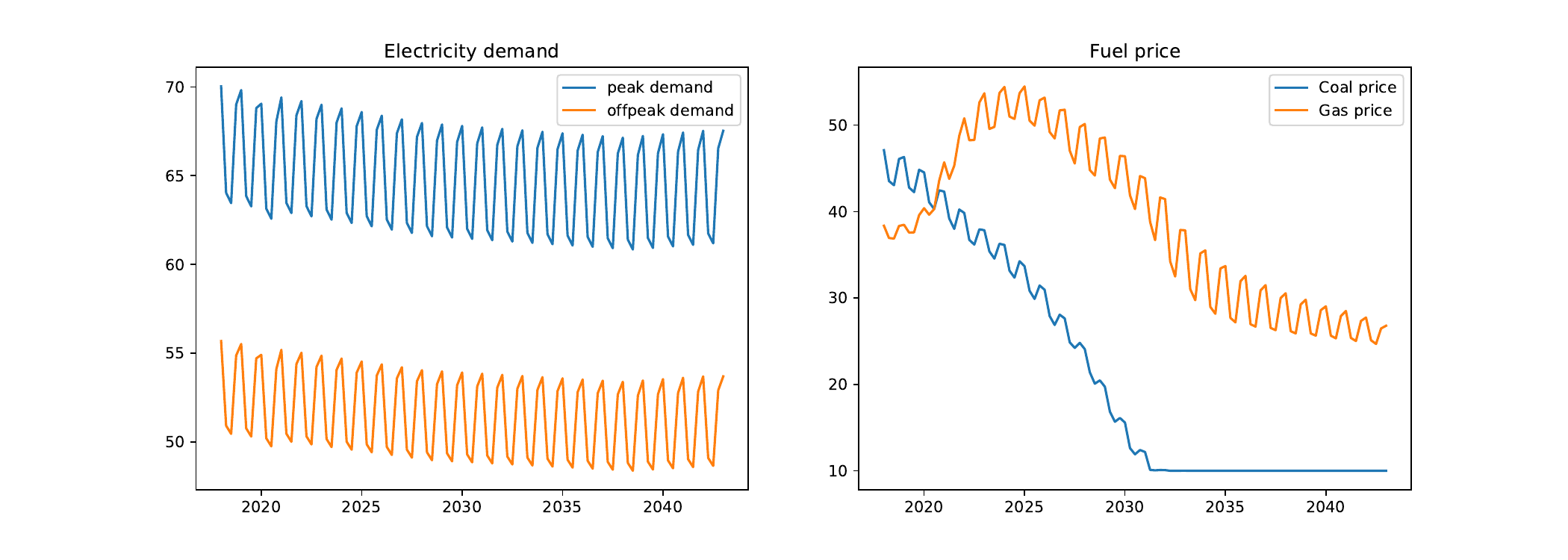}}
\caption{Left: exogenous peak and off-peak electricity demand. Right: gas and coal price trajectories, computed by the model.}
\label{demand.fig}
\end{figure}

\begin{figure}
\centerline{\includegraphics[width=\textwidth]{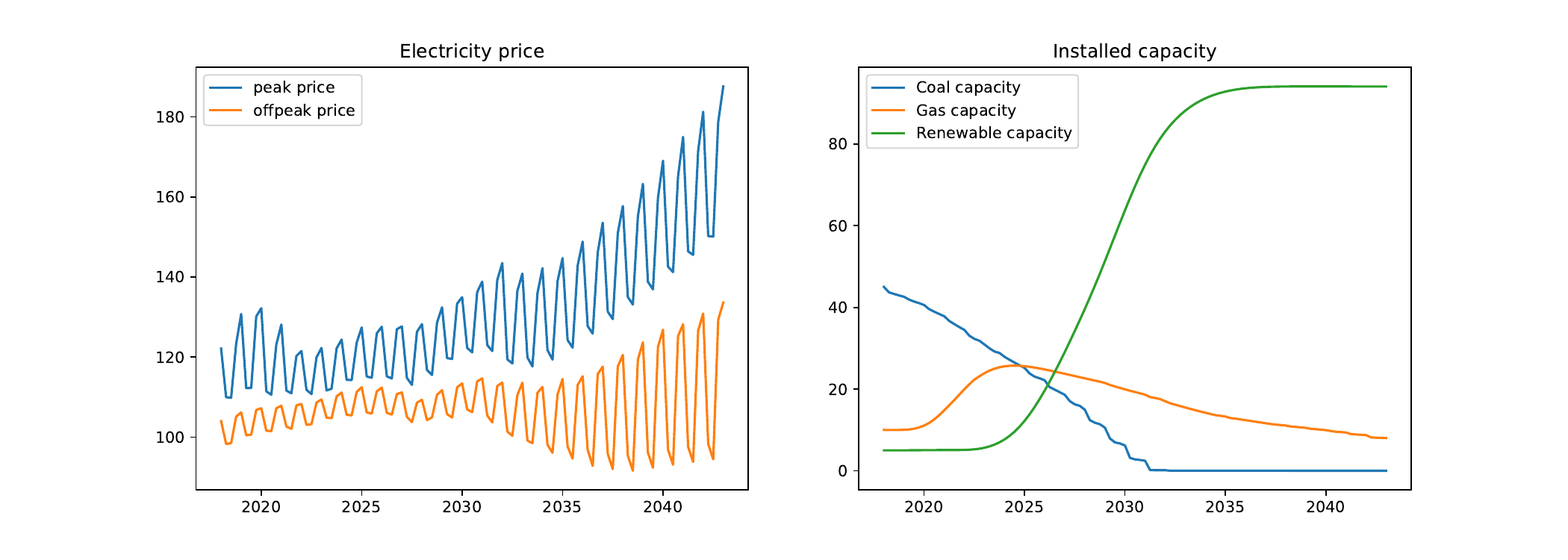}}
\caption{Left: peak and off-peak electricity price (euros per MWh). Right: installed capacity for the three technologies (GW).}
\label{price.fig}
\end{figure}

\begin{figure}
\centerline{\includegraphics[width=\textwidth]{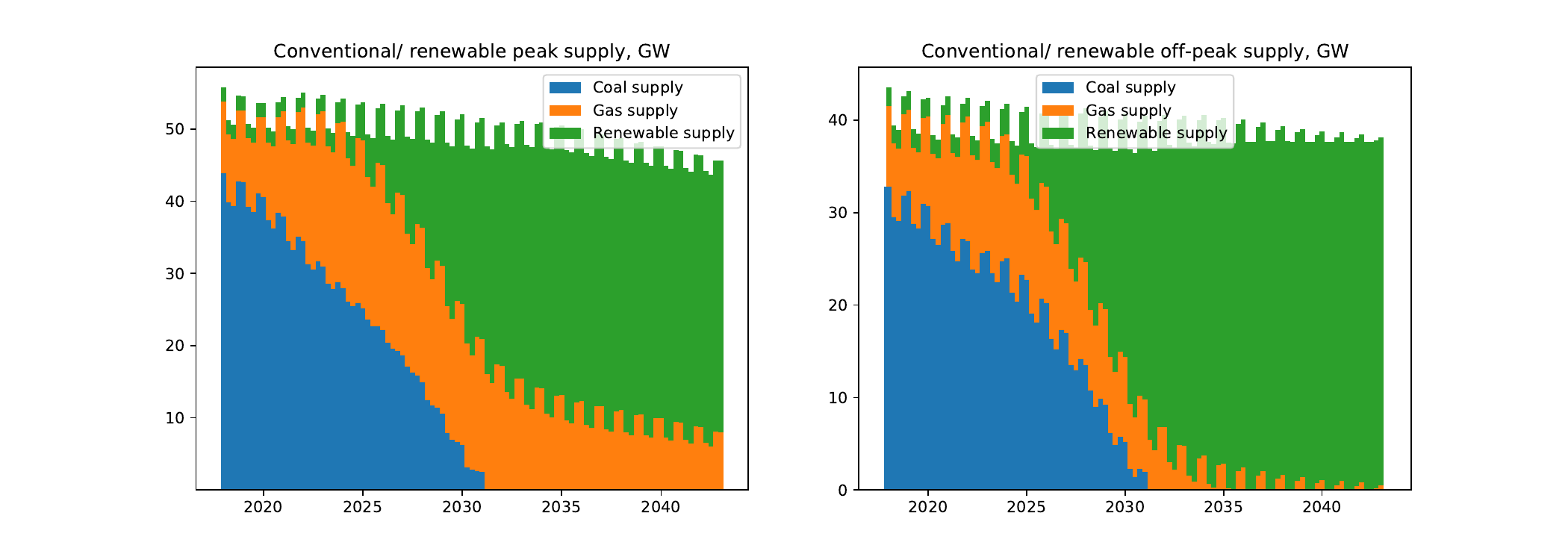}}
\caption{Peak (left) and off-peak electricity supply (GW).}
\label{supply.fig}
\end{figure}

\begin{figure}
\centerline{\includegraphics[width=\textwidth]{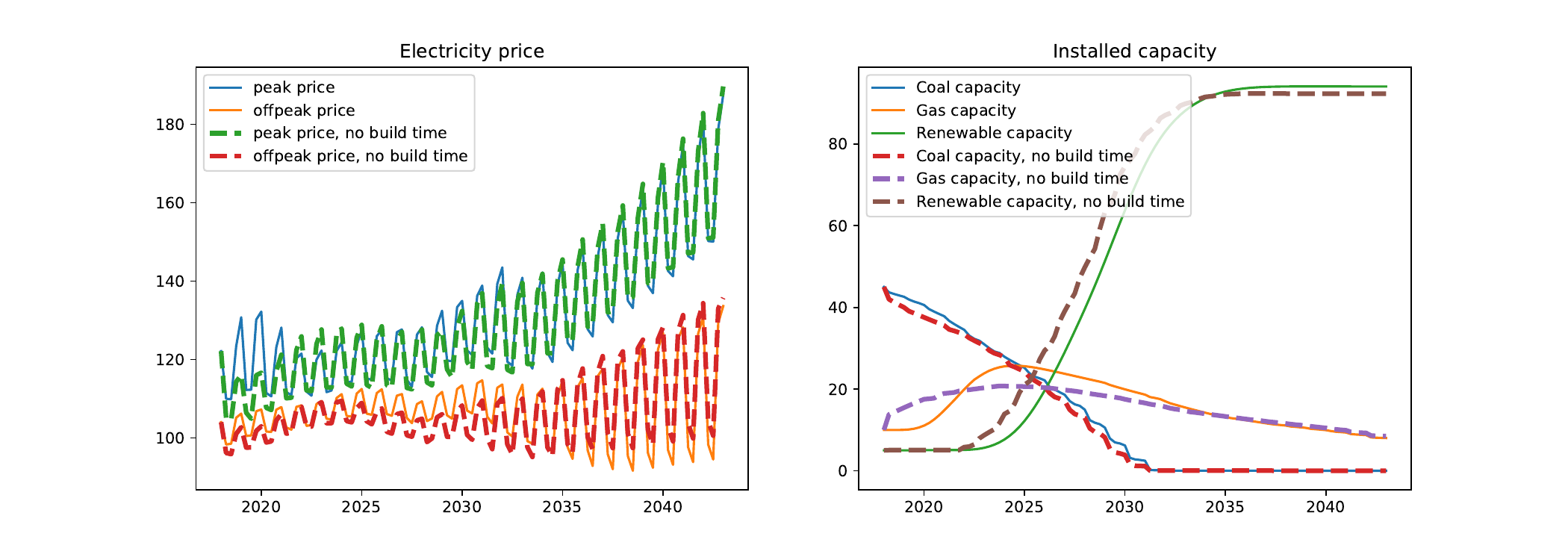}}
\caption{ Comparison of the model with zero and nonzero construction time. Left: peak and off-peak electricity price (euros per MWh). Right: installed capacity for the three technologies (GW).}
\label{timetobuild.fig}
\end{figure}
\bibliographystyle{chicago}
\bibliography{ms}

\appendix

\section{Proof of the main results}
We follow the structure of proofs of \cite{aid2020entry} with some important changes, owing to the facts that agents can both invest and divest and that the fuel prices are endogenously determined. 
\subsection{Preliminary lemmas}
All results in this section are shown under the assumptions of Theorem \ref{main.thm}. 
The first lemma establishes the compactness of the set $\mathcal R_i$. To this end, we associate the measure flow $(m^i_t)_{0\leq t\leq T}$ with a finite positive measure on $[0,T]\times \mathcal A\times \overline O_i$ defined by $m^i_t\, dt$, and we endow the set $\mathcal V_i$ with the topology of weak convergence of the associated measures. Similarly, the measure flow $(\hat m^i_t)_{0\leq t\leq T}$ is associated with a finite positive measure on $[0,T]\times \overline O_i$ defined by $\hat m^i_t\, dt$, and we endow the set $\widehat{\mathcal V}_i$ with the topology of weak convergence of the associated measures. The sets $\mathcal M_i$ and $\widehat{\mathcal M}_i$ are endowed with the standard weak convergence topology, and the set $\mathcal R_i(\hat \nu^i_0,\nu^i_0)$ is endowed with the product topology. 
\begin{lemma}
Fix $i\in\{1,\dots,N+\overline N\}$ and assume that the initial measures $\hat \nu^i_0$ and $\nu^i_0$ satisfy the assumptions \eqref{integr.eq}. 
Then, the set $\mathcal R_i(\hat \nu^i_0,\nu^i_0)$ is compact. 
\end{lemma}
\begin{proof}
The proof is similar to that of Theorem 2.13 in \cite{dumitrescu2021control}. Since $\mathcal R_i$ is a subset of a metrizable space, it is enough to show sequential compactness.  Consider then a sequence $(\hat \mu^{i,n}, \hat m^{i,n}, \mu^{i,n}, m^{i,n})_{n\geq 1} \subset \mathcal R_i$. First, use the test function $\hat u(t,x) = \int_t^T f(s)ds$, where $f$ is nonnegative bounded continuous. Using this test function in \eqref{constnash2} yields
\begin{align}
\int_0^T f(s) ds \int_{\overline O_i} \hat m^{i,n}(dx)\leq \int_0^T f(s) ds \int_{\overline O_i} \hat \nu^i_0(dx). \label{ubound}
\end{align}
A simple limiting argument (see Lemma 2.8 in \cite{dumitrescu2021control}) then shows that 
$$
\int_{ \overline O_i} \hat m^{i,n}(dx)\leq  \int_{\overline O_i} \hat \nu^i_0(dx)
$$
$t$-almost everywhere. Now use the test function $\hat u_k(t, x)=(T+1-t)\phi_k(x)$, where
$$
\phi_k(x)=\ln \left\{1+|x|^{3}\left(\frac{3 x^{2}}{5 k^{2}}-\frac{3|x|}{2 k}+1\right)\right\} \mathbf{1}_{|x| \leq k}+\ln \left\{1+\frac{k^{3}}{10}\right\} \mathbf{1}_{|x|>k}.
$$
For each $k\geq 1$, $\hat u_k\in C^{1, 2}_b([0, T]\times \bar{\mathcal{O}})$ and $\hat u_k$ is non-negative.   
Using this test function in \eqref{constnash2}, yields
\begin{multline*}
(T+1)\int_{ \overline{\mathcal  O}_i} \phi^k(x) \hat\nu^i_0(dx) \\+ \int_{[0,T] \times \overline{\mathcal  O}_i} \left\{-\phi^k(x) + (T+1-t)\mathcal L_i \phi^k(x)\right\}\hat m^{i,n}_t(dx)\, dt \geq 0.
\end{multline*}
Since there exists a constant $C\geq 0$ independent from $k$ such that $\phi_k'$ and $\phi''_k$ are bounded by $C$, together with \eqref{ubound}, we obtain that
\begin{align*}
  \int_{[0,T]\times \mathcal A \times \overline{\mathcal  O}_i} \phi^k(x) \hat m^{i,n}_t(dx)\, dt \leq (T+1)\int_{ \overline{\mathcal  O}_i} \phi^k(x) \hat\nu^i_0(dx) + C',
\end{align*}
for some constant $C'$ which does not depend on $n$ or $k$. Since, for each $x$, the sequence $(\phi_k(x))_{k\geq 1}$ is non-decreasing and converges to $\phi_k = \ln (1+|x|^3)$,  we conclude by monotone convergence and the assumption of the lemma, that 
\begin{align}
  \sup_n \int_{[0,T] \times \overline{\mathcal  O}_i} \phi(x) \hat m^{i,n}_t(dx)\, dt \leq (T+1)\int_{ \overline{\mathcal  O}_i} \phi(x) \hat\nu^i_0(dx) + C'<\infty,\label{boundmhat}
\end{align}
which proves the tightness of the sequence $(\hat m^{i,n})_{n\geq 1}$. 

In a similar way, the condition \eqref{constnash2} with the same test function yields a bound on $\hat \mu^{i,n}$:
\begin{align*}
 \sup_n\int_{[0,T] \times \overline{\mathcal  O}_i} \phi(x)\hat \mu^{i,n}(dt,dx)\leq (T+1)\int_{ \overline{\mathcal  O}_i} \phi(x) \hat\nu^i_0(dx)+ C'<\infty,
\end{align*}
which proves the tightness of this family of measures. 

Now, using the test function $u(t,a,x) = \int_t^T f(s) ds$  in the condition \eqref{constnash1}, and remarking that 
$$
\int_{[0,T]\times \mathcal A \times \overline{\mathcal  O}_i} \hat \mu^{i,n}(dt,da,dx) = \int_{\overline{\mathcal  O}_i} \hat \nu^{i}_0(dx),
$$
we obtain that 
$$
\int_{\mathcal A \times \overline O_i}  m^{i,n}(da,dx)\leq  \int_{\overline O_i} \hat \nu^i_0(dx) + \int_{\mathcal A \times \overline O_i} \nu^i_0(da, d x)
$$
$t$-almost everywhere. Finally, using the test function $u_k(t,a,x) = (T+1-t)\phi_k(x)$ in \eqref{constnash1} shows, in a similar way as above, that 
\begin{align*}
 & \sup_n \int_{[0,T]\times \mathcal A \times \overline{\mathcal  O}_i} \phi(x) m^{i,n}_t(da,dx)\, dt \leq (T+1)\int_{ \overline{\mathcal  O}_i} \phi(x) \hat\nu^i_0(dx) \\&\qquad \qquad +(T+1)\int_{\mathcal A \times \overline O_i} \phi(x)\nu^i_0(da, d x)+ C^{\prime\prime}<\infty,\\
 &\sup_n\int_{[0,T]\times \mathcal A \times \overline{\mathcal  O}_i} \phi(x) \mu^{i,n}(dt,da,dx)\leq (T+1)\int_{ \overline{\mathcal  O}_i} \phi(x) \hat\nu^i_0(dx)\\ &\qquad \qquad +(T+1)\int_{ \mathcal A\times \overline{\mathcal  O}_i} \phi(x) \nu^i_0(da,dx)+ C^{\prime\prime}<\infty,
\end{align*}
which proves the tightness of the corresponding sequences. Together, all our estimates prove that the sequence $(\hat \mu^{i,n}, \hat m^{i,n}, \mu^{i,n}, m^{i,n})_{n\geq 1}$ has a subsequence which converges to a 4-uplet $(\hat \mu^{i}, \hat m^{i}, \mu^{i}, m^{i})$. Since by construction the set $\mathcal R_i(\hat\nu^i_0,\nu^i_0)$ is closed with respect to weak convergence, $(\hat \mu^{i}, \hat m^{i}, \mu^{i}, m^{i})\subset \mathcal R_i$ and the proof is complete. 
\end{proof}
We now prove a technical result regarding a bounded variation property of some specific functionals, which will be used throughout the proofs.
\begin{lemma}[A bounded variation property]\label{BV} Fix $i=N+1, \ldots, N+\bar{N}$. Let $h \in C_b^{1,1,2}([0,T] \times \mathcal{A} \times \overline{\mathcal{O}}_{i})$, $\hat h \in C_b^{1,2}([0,T]  \times \overline{\mathcal{O}}_{i})$ and $(m^{i},\hat{m}^{i},\mu^{i},\hat{\mu}^{i}) \in \mathcal{R}_i$. Then, for every $\psi \in C^{1}([0,T]),$
\begin{align*}
\int_0^T \psi^\prime(t)\left(\int_{\mathcal{A} \times \overline{\mathcal{O}}_{i}}h(t,a,x)m^{i}_t(da,dx)\right)dt \leq C|\!|\psi|\!|_{\infty}.
\end{align*}
and
\begin{align*}
\int_0^T \psi^\prime(t)\left(\int_{\overline{\mathcal{O}}_{i}}h(t,x)\hat{m}^{i}_t(ddx)\right)dt \leq C|\!|\psi|\!|_{\infty}.
\end{align*}
For $i=1,\ldots, N$, a similar property holds with functions $h$ which do not depend on $x$.
\end{lemma}
\begin{proof}
Fix $i=N+1,\ldots, N+\bar{N}$.  Consider the test functions $u(t,a,x)=-\psi(t)h(t,a,x)$ and $\hat u(t,x)=-\psi(t)h(t,x)$. Using again the constraints \eqref{constnash1}-\eqref{constnash2}, the boundedness of $h$ and its derivatives, the boundedness of the coefficients of the diffusion processs $S$, together with the boundedness of  the flow of measures $m^{i}$ and $\hat{m}^{i}$ proved in the previous lemma, the result follows.
\end{proof}

The following lemma establishes the existence of measures maximizing the gain functions for \emph{fixed price trajectories}.
\begin{lemma}
Fix $i\in\{1,\dots,N+\overline N\}$ and assume that the initial measures $\hat \nu^i_0$ and $\nu^i_0$ satisfy the assumptions \eqref{integr.eq}. 
Assume that the peak demand $D^p$, the off-peak demand $D^{op}$, the carbon price $P^C$ trajectories, the electricity price trajectories $P^p$ and $P^{op}$ and the fuel price trajectories $P^1,\dots,P^K$ are fixed and have finite variation on $[0,T]$. Then, for each $i$, there exist a 4-uplet
$$
(\hat \mu^i, (\hat m^i_t)_{0\leq t\leq T},\mu^i, ( m^i_t)_{0\leq t\leq T}) \in \mathcal R_i(\hat \nu^i_0,\nu^i_0),
$$
which maxmizes the functional \eqref{optconv.mfg} (if $i\in \{1,\dots,N\}$ or \eqref{optren.mfg} (if $i\in\{N+1,\dots,N+\overline N\}$. 
\end{lemma}
\begin{proof}
The proof uses the Lipschitz property of the gain function $G_i$ and the approximation of the price processes with uniformly bounded continuous mappings in $L^1([0,T])$. Since it is very similar to the proof of Lemma 3.2 in \cite{aid2020entry}, the details are omitted here to save space. 
\end{proof}
The following lemma establishes the properties of the price process. 
\begin{lemma}\label{conv}
${}$
\begin{itemize}
\item[i.]
Let 
$$
(\hat \mu^i, (\hat m^i_t)_{0\leq t\leq T}, \mu^i, ( m^i_t)_{0\leq t\leq T}) \in \mathcal R_i,\quad i=1,\dots,N+\overline N,
$$
and assume that the peak demand $D^p$, the off-peak demand $D^{op}$ and the carbon price $P^C$ trajectories have bounded variation on $[0,T]$. Then, the electricity price trajectories $P^p$ and $P^{op}$ and the fuel price trajectories $P^1,\dots,P^K$ trajectories have bounded variation on $[0,T]$ as well. 
\item[ii.] Let $$
\left((\hat{\mu}_n^1,\hat{m}_n^1,\mu_n^1,m_n^1), \ldots, (\hat{\mu}^{N+\overline N}_n,\hat{m}^{N+\overline N}_n,\mu^{N+\overline N}_n,m^{N+\overline N}_n)\right) \in \mathcal{R}_1 \times \ldots \times \mathcal{R}_{N+\bar{N}},
$$
be a sequence converging to  \begin{multline*}
\left((\hat{\mu}^{1,\star},\hat{m}^{1,\star},\mu^{1,\star},m^{1,\star}), \ldots, (\hat{\mu}^{N+\overline N,\star},\hat{m}^{N+\overline N,\star},\mu^{N+\overline N,\star},m^{N+\overline N,\star})\right) \\\in \mathcal{R}_1 \times \ldots \times \mathcal{R}_{N+\bar{N}}.
\end{multline*}
Assume that the peak demand $D^p$, the off-peak demand $D^{op}$ and the carbon price $P^C$ trajectories have bounded variation on $[0,T]$. Then there exists a subsequence $(n_l)$ such that $P_t^{p,n_l}$, $P_t^{op,n_l}$ and $P_t^{k,n_l}$  converge in $L^1([0,T])$ to $P_t^{p,\star}$, $P_t^{op,\star}$ and $P_t^{k,\star}$.
\end{itemize}
\end{lemma}
\begin{proof}
The proof follows that of Lemma 3.3 in \cite{aid2020entry} with some important changes and additions due to the fact that the price vector is now multidimensional (it includes not only the electricity price but also the endogenous fuel prices). We shall provide details of the changes and refer to \cite{aid2020entry} for the parts of the proof, which are very similar to that reference. 
\paragraph{Part i.} The first step is to show that the peak residual demand $((D^p_t - R_t)^+)_{0\leq t\leq T}$ and the off-peak residual demand $((D^{op}_t - R_t)^+)_{0\leq t\leq T}$ have bounded variation on $[0,T]$. This is done similarly to the first step of \cite{aid2020entry} using Lemma \ref{BV}.

In the second step we prove that the fuel prices $P^1,\dots,P^K$ are bounded. To this end, observe that the function $G_t$ defined in the proof of Proposition \ref{price.prop}, equals zero when all prices are zero, and admits the following lower bound:
\begin{align*}
G_t (P^p, P^{op},P^1,\dots,P^K) \geq -P^*(c_p D^p_t + c_{op}D^{op}_t) + \sum_{k=1}^K \overline\Phi_k(P^k).
\end{align*}
Since the functions $\overline\Phi_k$ for $k=1,\dots,K$ are positive and strongly convex, it follows that $P^1_t,\dots,P^K_t$ admit an upper bound $\overline P$ (which may depend on the demand trajectories but they are fixed in this proof).  Without loss of generality (by increasing the constant $\overline P$ if needed), we assume that the trajectories of $|P^p - f_i (P^k + e_k P^C)|$  and  $|P^{op}- f_i (P^k + e_k P^C)|$  are also bounded from above by $\overline P$ for each $k=1,\dots,K$. 

In the third step we introduce the ``discretized'' offer functions. Namely for a fixed $n$, define
$$
F^{i,n}_{q}(t) = \int_{\mathcal A\times \overline{\mathcal O}_i} m^i_t(da,dx) \lambda_i(a) F_i(q\overline P/n - x)
$$
for $q=-n,\dots,n$. We can then show, by following the arguments in step 2 of \cite{aid2020entry}, that these functions have bounded variation on $[0,T]$, uniformly on $n,q$, in the sense that for every $n$ and every family of $C^1$ functions $\psi_{q}:[0,T]\to \mathbb R$, $q=-n,\dots,n$, we have
$$
\sum_{q=-n}^n \int_0^T F^{i,n}_{q}(t)\psi'_{q}(t) dt \leq C \max_{0\leq t\leq T} \sum_{q=-n}^n |\psi_{q}(t)|,
$$
where the constant $C$ does not depend on $q,n$.

In the fourth step we introduce the mollified offer functions. Let $\rho$ be a mollifier supported on $[-1,1]
$, set $\rho_m(x) = m\rho(mx)$ and define
$$
F^{i,n,m}_{q}(t):= \rho_m \star F^{i,n}_{q}(t),
$$
where $F^{i,n}_{q}$ is extended by zero value outside the interval $[0,T]$.  Then we can show, by following the arguments in step 3 of \cite{aid2020entry} that
$$
\int_0^T \max_{-n\leq q\leq n} \Big|\frac{d}{dt}F^{i,n,m}_{q}(t) \Big|dt \leq C,
$$
for a different constant $C$. 

In the final step, we introduce sequences of mappings $\Theta^p_n,\Theta^{op}_n,\Theta^1_n,\dots,\Theta^K_n$, $n\geq 1$, such that $\Theta^p_n$ and $\Theta^{op}_n$ approximate the peak and off-peak electricity price, and $\Theta^1_n,\dots,\Theta^K_n$ approximate the fuel prices in the following way: the peak price is approximated by the expression
\begin{align}
P^{p,n}_t  := \Theta^p_n((D^p_t-R_t)^+ , (D^{op}_t-R_t)^+, P^C_t, (F^{i,n}_{q}(t))^{i=1,\dots,N}_{q = -n,\dots,n}),\label{approxprice}
\end{align}
and similarly for the other prices. Lemma \ref{mapping.lm} shows that one may ideed construct such mappings with required properties. To complete the proof we then follow the arguments in step 4 of \cite{aid2020entry}. 
\paragraph{Part ii.} 
This part follows closely the proof of Part ii.~of Lemma 3.3 in \cite{aid2020entry}. 
\end{proof}
\begin{lemma}\label{mapping.lm}
  There exist sequences of mappings (the arguments in parentheses are the same for all mappings but we only spell them out for the first one): $$
  \Theta^p_n(u,v,w,(\xi^i_{q})^{i=1,\dots,N}_{q=-n,\dots,n}),\Theta^{op}_n,\Theta^1_n,\dots,\Theta^K_n,\ n\geq 1
  $$
  with the following properties:
  \begin{enumerate}
  \item $\Theta^p_n\leq P^*,\  \Theta^{op}_n\leq P^*,\ \Theta^1_n\leq \overline P,\dots,\Theta^K_n\leq \overline P$.
  \item There exists a constant $C$ such that $|P^{p,n}_t - P^p_t|\leq \frac{C}{n}$,   $|P^{op,n}_t - P^{op}_t|\leq \frac{C}{n}$,  $|P^{k,n}_t - P^k_t|\leq \frac{C}{n}$ for $k=1,\dots,K$, where the approximate price $P^{p,n}_t $ is defined in \eqref{approxprice} and the other approximate prices are defined similarly.
  \item  $\Theta^p_n,\Theta^{op}_n,\Theta^1_n,\dots,\Theta^K_n$ are differentiable, and their derivatives satisfy, on any compact set, 
    \begin{align*}
      &\Big|\frac{\partial \Theta^\star_n}{\partial u}\Big| \leq C , \quad \Big|\frac{\partial \Theta^\star_n}{\partial v}\Big| \leq C,\quad \Big|\frac{\partial \Theta^\star_n}{\partial w}\Big| \leq C \\
      &\text{and} \sum_{i=1,\dots,N, q = -n,\dots, n} \Big|\frac{\partial \Theta^\star_n}{\partial \xi^i_{q}}\Big| \leq C,
    \end{align*}
    where $C$ is a constant which does not depend on $n$, where $\star$ stands for $p$, $op$ or $1,\dots,K$.  
\end{enumerate} 
\end{lemma}
\begin{proof}
  To define these mappings, we consider the linear interpolation maps
  $$
\mathcal L_n(y,\xi_{-n},\dots,\xi_n) = ((q_y+1) \overline P/n - y)\xi_{q_y} + (y-q_y \overline P/n)\xi_{q_y + 1},
$$
for $y\in (-\overline P, \overline P)$, 
where $q_y = \max\{q:y\overline P/n<y\}$. From these interpolation mappings we define
\begin{align*}
{F}^{k,n}(x,y,(\xi^i_q)^{i=1,\dots,N}_{q=-n,\dots,n})&= \sum_{i: k(i) = k} \mathcal L_n(x- f_i y, \xi^i_{-n},\dots,\xi^i_n)\\
\Psi^{k,n}(x,y,(\xi^i_q)^{i=1,\dots,N}_{q=-n,\dots,n})&= \sum_{i: k(i) = k} f_i \mathcal L_n(x- f_i y, \xi^i_{-n},\dots,\xi^i_n).
\end{align*}
Let $\overline\Theta^p_n,\overline\Theta^{op}_n,\overline\Theta^1_n,\dots,\overline\Theta^K_n$, be  the solution to
\begin{align}
&  c_p \Psi^{k,n}(\theta^p, \theta^k)+c_{op} \Psi^{k,n}(\theta^{op}, \theta^k) = \Phi_k (\theta^k - e_k w),\quad k=1,\dots,K,\label{eqprice1}
\end{align}
where $\Phi_k$ is extended by zero to negative values, and
\begin{align}
  &  u =  F_0(\theta^p) + \sum_{k=1}^K\widehat{ F}^{k,n}(\theta^p, \theta^k) \ \text{or}\ u >  F_0(\theta^p) + \sum_{k=1}^K\widehat{F}^{k,n}(\theta^p, \theta^k) \ \text{and}\ \theta^p = P^*,\label{eqprice2}\\
  &  v =  F_0(\theta^{op}) + \sum_{k=1}^K\widehat{F}^{k,n}(\theta^{op}, \theta^k)\ \text{or}\ v >  F_0(\theta^{op}) + \sum_{k=1}^K\widehat{ F}^{k,n}(\theta^{op}, \theta^k)\ \text{and}\ \theta^{op} = P^*,\label{eqprice3}
\end{align}  
where we omitted the arguments $\xi^i_q$ to save space. The value of the mappings $\Theta^p_n,\Theta^{op}_n,\Theta^1_n,\dots,\Theta^K_n$ at the point $(u,v,(\xi^i_{q})^{i=1,\dots,N}_{q=-n,\dots,n})$ is defined as follows: 
$$
\Theta^p_n = \overline \Theta^p_n,\quad \Theta^{op}_n = \overline \Theta^{op}_n,\quad \Theta^1_n = \overline\Theta^1_n-e_1 w,\dots,\Theta^K_n = \overline\Theta^K_n-e_K w.
$$

To show that $\Theta^p_n,\Theta^{op}_n,\Theta^1_n,\dots,\Theta^K_n$ are uniquely determined, define 
$$
\overline{\mathcal L}_n(y,\xi_{-n},\dots,\xi_n):=\int_0^y \mathcal L_n(y',\xi_{-n},\dots,\xi_n) dy',
$$
and consider the mapping
\begin{align*}
&G^n(\theta^p, \theta^{op}, \theta^1,\dots,\theta^K):=\sum_{k=1}^K \sum_{i:k(i)=k}\{c_p \overline{\mathcal L}_n(\theta^p-f_i \theta^k)+c_{op} \overline{\mathcal L}_n(\theta^{op}-f_i \theta^k)\} \\ & \qquad +c_p G_0(\theta^p) + c_{op}G_0(\theta^{op})- c_p u \theta^p - c_{op} v\theta^{op}  + \sum_{k=1}^K \overline \Phi_k(\theta^k),
\end{align*}
where we used some notation of Proposition \ref{price.prop}. This mapping is strongly convex (due to strong convexity of $G_0$ and $\overline \Phi_1,\dots,\overline \Phi_K$) and the vector $(\overline\Theta^p_n,\overline\Theta^{op}_n,\overline\Theta^1_n,\dots,\overline\Theta^K_n)$ is its unique minimizer. 

Next, observe that by construction, 
$$
F^{i,n}_{q_y} \leq \mathcal L_n(y,F^{i,n}_{-n}(t),\dots,F^{i,n}_n(t))\leq  F^{i,n}_{q_y+1} 
$$
It follows that 
\begin{multline*}
\int_{\mathcal A\times \overline{\mathcal O}_i} m^i_t(da,dx) \lambda_i(a) F_i(y-\overline P/n - x)\leq 
 L_n(y,F^{i,n}_{-n}(t),\dots,F^{i,n}_n(t))\\\leq \int_{\mathcal A\times \overline{\mathcal O}_i} m^i_t(da,dx) \lambda_i(a) F_i(y+\overline P/n - x).
 \end{multline*}
 Integrating the upper bound, we get: 
 \begin{align*}
 \overline{\mathcal L}_n(y) &\leq \int_{\mathcal A\times \overline{\mathcal O}_i} m^i_t(da,dx) \lambda_i(a) \int_0^y F_i(y'+\overline P/n - x) dy' \\
 & \leq \int_{\mathcal A\times \overline{\mathcal O}_i} m^i_t(da,dx) \lambda_i(a) G_i(y-x) + \frac{C}{n},
 \end{align*}
 for some constant $C$, by the smoothness of $F_i$ and integrability of $m^i$. Proceeding similarly for the lower bound and summing up the terms, we finally get (with the notation of Proposition \ref{price.prop}): 
 \begin{multline*}
\Big| \sum_{k=1}^K \sum_{i:k(i)=k}\{c_p \overline{\mathcal L}_n(\theta^p-f_i \theta^k)+c_{op} \overline{\mathcal L}_n(\theta^{op}-f_i \theta_k)\} \\- \sum_{k=1}^K \{c_p  G^k_t(\theta^p,\theta^k)-c_{op} G^k_t(\theta^{op},\theta^k)\}
\Big|\leq \frac{C}{n}.  
\end{multline*}
Property 2 of the lemma now follows from strong convexity. 

We now concentrate on property 3. To save space, we denote simply $\pmb \theta:= (\theta^p,\theta^{op},\theta^1,\dots,\theta^K)$. 
Let $\pmb\Theta'_n$ be the minimizer of $G_n(\pmb\theta) -c_p(u'-u)\theta^p$. It is characterized by the equation
$$
\nabla G_n(\pmb\Theta'_n) = c^p(u'-u)\mathbf{e}^p,
$$
where $\mathbf{e}^p$ is the $K+2$-dimensional vector with $1$ in the first position and $0$ elsewhere. Differentiating both sides using the smoothness of $G^n$, we get: 
$$
\nabla^2 G^n(\pmb\Theta'_n)\frac{\partial\pmb\Theta'_n}{\partial u'} = c^p \pmb e^p. 
$$
By strong convexity, $\nabla^2 G^n(\pmb\Theta'_n)$ is invertible and bounded from below, whence the boundedness of $\frac{\partial\pmb\Theta'_n}{\partial u'}$. The boundedness of the derivative with respect to $v$ is obtained in a similar fashion. To analyze the derivative with respect to $w$, in a similar fashion, we may compute:
$$
\nabla^2 G^n(\pmb\Theta'_n)\frac{\partial\pmb\Theta'_n}{\partial w'} = \mathbf e_k \sum_{k=1}^K\Phi'_k (\theta^k-e_k w). 
$$
Since the RHS is bounded on compact set, we conclude that the derivative with respect to $w$ is bounded. 

Finally analyze the derivatives with respect to $\xi^i_q$, to save space, we denote by $\pmb\xi$ the family 
$(\xi^i_{q})^{i=1,\dots,N}_{q=-n,\dots,n})$, and denote by $\pmb\Theta'_n$ the minimizer of $G_n$ with $\pmb\xi$ replaced by $\pmb\xi'$. Then, proceding as above and
differentiating both sides with respect to $\xi^{i\prime}_q$, we get:
\begin{multline*}
\nabla^2 G_n(\pmb\Theta'_n)\frac{\partial \Theta'_n}{\partial \xi^{i}_q} + (\mathbf{e}^p - f_i\mathbf{e}^{k(i)})c_p\{ ((q_{y_p}+1) \overline P/n - y_p)\mathbf 1_{q=q_{y_p}} + (y_p-q_{y_p} \overline P/n)\mathbf 1_{q=q_{y_p}+1}\}\\+(\mathbf{e}^{op} - f_i\mathbf{e}^{k(i)})c_{op}\{ ((q_{y_{op}}+1) \overline P/n - y_{op})\mathbf 1_{q=q_{y_{op}}} + (y_{op}-q_{y_{op}} \overline P/n)\mathbf 1_{q=q_{y_{op}}+1}\},
\end{multline*}
where $y_p = \theta^{p\prime} - f_i \theta^{k\prime}$ and $y_{op} = \theta^{op\prime} - f_i\theta^{k\prime}$. By strong convexity we obtain the boundedness of $\frac{\partial \Theta'_n}{\partial \xi^{i}_q}$. In addition, for fixed $i$, among the derivatives $\frac{\partial \Theta'_n}{\partial \xi^{i}_q}$, $q=-n,\dots,n$, only at most four are nonzero (because of the indicator functions). Thus, sum of derivatives is bounded uniformly on $n$. 
\end{proof}  
\subsection{Proof of Theorem \ref{main.thm}}
For $i=1,\ldots, N+\bar{N}$, denote by ${\boldsymbol{m}^{i}}$ the family $(\hat{\mu}^{i},(\hat{m}^{i}_t)_{0 \leq t \leq T},\mu^{i},(m^{i}_t)_{0 \leq t \leq T})$ and ${\boldsymbol{{m}}}=(\boldsymbol{m}^1,\boldsymbol{m}^2,\ldots, \boldsymbol{m}^{N+\bar{N}})$. With this notation, the condition of the lemma simplifies to $\boldsymbol{m}_n\to \boldsymbol{m}^*$. 
We also define, for $i=1, \ldots, N$,
\begin{align*}
&\Psi^{i}(\boldsymbol{m}^{i},\boldsymbol{m})=\int_{[0,T]\times \mathcal A \times \overline{\mathcal O}_i} m^i_t(da, dx) e^{-\rho
  t}\lambda_i(a)\left[c_p G_i(P^p_t(\boldsymbol{m}) - e_i P^C_t-f_i P^{k(i)}_t(\boldsymbol{m}) - x )\right.\nonumber \\  & \left.+c_{op} G_i(P^{op}_t(\boldsymbol{m}) - e_i P^C_t-f_i P^{k(i)}_t(\boldsymbol{m}) - x )-\kappa_i \right]dt \notag\\&- K_i \int_{[0,T] \times \overline{\mathcal O}_i} \hat \mu^i(dt, dx) e^{-(\rho+\gamma_i)
  t} +
  \widetilde K_i \int_{[0,T]\times \mathcal A \times \overline{\mathcal O}_i}  \mu^i(dt, da, dx) e^{-(\rho+\gamma_i) t}
\end{align*}
  and, for $i=N+1,\ldots,N+\bar{N}$,
\begin{align*}
  &\bar{\Psi}^{i}(\boldsymbol{m}^i,\boldsymbol{m})=\int_{[0,T]\times \mathcal A \times \overline{\mathcal O}_i} m^i_t(da, dx) e^{-\rho
  t}\lambda_i(a)\left[
  (c_p P^p_t(\boldsymbol{m})+c_{op}P^{op}_t(\boldsymbol{m})) x - {\kappa_i}\right] dt \nonumber \\&- K_i \int_{[0,T] \times \overline{\mathcal O}_i} \hat \mu^i(dt,  dx) e^{-(\rho+\gamma_i)
  t}
  +
  \widetilde K_i \int_{[0,T]\times \mathcal A \times \overline{\mathcal O}_i}  \mu^i(dt, da, dx) e^{-(\rho+\gamma_i) t}.
\end{align*}
Consider the set valued map
$$\Theta: \mathcal{R}(\hat\nu^1_0,\nu^1_0)\times \ldots \times \mathcal{R}(\hat\nu^{N+\bar{N}}_0,\nu^{N+\bar{N}}_0)\to 2^{\mathcal{R}(\hat\nu^1_0,\nu^1_0)\times \ldots \times \mathcal{R}(\hat\nu^{N+\bar{N}}_0,\nu^{N+\bar{N}}_0)},$$
which is given by
\begin{multline}
\Theta:{\boldsymbol{m}} \mapsto \otimes_{i=1}^{N}\text{argmax}_{\boldsymbol{m}^{i} \in \mathcal R_i(\hat\nu^i_0,\nu^i_0)} \Psi^{i}(\boldsymbol{m}^{i},\boldsymbol{{m}})\\ \otimes_{i=N+1}^{N+\bar N}\text{argmax}_{\boldsymbol{m}^{i} \in \mathcal R_i(\hat\nu^i_0,\nu^i_0)} \bar{\Psi}^{i}(\boldsymbol{m}^{i},\boldsymbol{{m}}).
 \end{multline}
 To prove the existence of an equilibrium ${\boldsymbol{m}}^\star \in \mathcal{R}(\hat\nu^1_0,\nu^1_0)\times \ldots \times \mathcal{R}(\hat\nu^{N+\bar{N}}_0,\nu^{N+\bar{N}}_0)$ by using the  Fan-Glicksberg fixed
point theorem, we only need to show that $\Theta$ has closed graph, which is defined as 
$$\text{Gr}(\Theta)=\left\{({\boldsymbol{m}},{\boldsymbol{m}}') \in \left(\mathcal{R}(\hat\nu^1_0,\nu^1_0)\times \ldots \times \mathcal{R}(\hat\nu^{N+\bar{N}}_0,\nu^{N+\bar{N}}_0)\right)^2: {\boldsymbol{m}}' \in \Theta({\boldsymbol{m}})\right\}.$$
To show that $\text{Gr}(\Theta)$ is closed, we have to prove that for any sequence $({\boldsymbol{m}}_{1,n}, {\boldsymbol{m}}_{2,n})$ which weakly converges to $(\widetilde{{\boldsymbol{m}}}_1, \widetilde{{\boldsymbol{m}}}_2) \in \left(\mathcal{R}(\hat\nu^1_0,\nu^1_0)\times \ldots \times \mathcal{R}(\hat\nu^{N+\bar{N}}_0,\nu^{N+\bar{N}}_0)\right)^2$ such that ${\boldsymbol{m}}_{2,n} \in \Theta({\boldsymbol{m}}_{1,n})$, we have that $\widetilde{{\boldsymbol{m}}}_2 \in \Theta(\widetilde{{\boldsymbol{m}}}_1)$.

To prove this result, we fix first $i=N+1,\ldots N+\bar{N}$ and show that
\begin{align*}
& \lim_{n \rightarrow \infty} \int_{[0,T]\times \mathcal A \times \overline{\mathcal O}_i} m^i_{2,n,t}(da, dx) e^{-\rho
  t}\lambda_i(a)\left[
  (c_p P^p_t({\boldsymbol{m}}_{1,n})+c_{op}P_t^{op}({\boldsymbol{m}}_{1,n})) x - {\kappa_i}\right] dt \nonumber \\&- K_i \int_{[0,T]\times \overline{\mathcal O}_i} \hat \mu^i_{2,n}(dt, dx) e^{-(\rho+\gamma_i)
  t}
  +
  \widetilde K_i \int_{[0,T]\times \mathcal A \times \overline{\mathcal O}_i}  \mu^i_{2,n}(dt, da, dx) e^{-(\rho+\gamma_i) t} \nonumber \\
  & =  \int_{[0,T]\times \mathcal A \times \overline{\mathcal O}_i} \widetilde{m}^i_{2,t}(da, dx) e^{-\rho
  t}\lambda_i(a)\left[
  (c_p P^p_t(\widetilde{{\boldsymbol{m}}}_{1})+c_{op}P_t^{op}(\widetilde{{\boldsymbol{m}}}_{1})) x - {\kappa_i}\right] dt \nonumber \\&- K_i \int_{[0,T] \times \overline{\mathcal O}_i} \widetilde{\hat \mu}^i_2(dt, dx) e^{-(\rho+\gamma_i)
  t}
  +
  \widetilde K_i \int_{[0,T]\times \mathcal A \times \overline{\mathcal O}_i}  \widetilde{\mu}^i_2(dt, da, dx) e^{-(\rho+\gamma_i) t}. 
\end{align*}
Using Lemma \ref{BV}, we can prove that the total variation of the map $t \mapsto h(t)$, with $h(t)$ given by $h(t)=\int_{  \bar{\mathcal{O}}_i \times \mathcal{A}} x \lambda_i(a)m^i_{2,n,t}(da,dx)$ is uniformly bounded with respect to $n$, and therefore by Theorem 3.23 in \cite{ambrosio2000functions}, we get that there exists a subsequence $(n_k)$ such that the maps $\int_{\mathcal{A} \times \bar{\mathcal{O}}_i} x \lambda_i(a)m^i_{2,n_k,\cdot}(da,dx)$ converges in $L^1([0,T])$ to some limit which can be identified due to the weak convergence with $\int_{\mathcal{A} \times \bar{\mathcal{O}}_i} x \lambda_i(a)\widetilde{m}^i_{2,\cdot}(da,dx)$. Moreover, in view of Lemma \ref{conv} item ii., $P_t^p({\boldsymbol{m}}_{1,n})$ and $P_t^{op}({\boldsymbol{m}}_{1,n})$ converge (up to a subsequence) to $P^{p}_t(\widetilde{{\boldsymbol{m}}}_1)$ and $P^{op}_t(\widetilde{{\boldsymbol{m}}}_2)$. Since both terms are bounded, we derive that that the integral of their product converges too. The convergence of the integrals with respect to $\hat{\mu}^{i}_{2,n}$ and ${\mu}^{i}_{2,n}$ follows by weak convergence of the measures.

It remains to show that, for $i=1, \ldots, N,$
\begin{align*}
&\lim_{n \mapsto \infty} \int_{[0,T]\times \mathcal A \times \overline{\mathcal O}_i} m^i_{2,n,t}(da, dx) e^{-\rho
  t}\lambda_i(a)\left[c_p G_i(P^p_t({\boldsymbol{m}}_{1,n}) - e_i P^C_t-f_i P^{k(i)}_t({\boldsymbol{m}}_{1,n}) - x )\right.\nonumber \\  & \left.+c_{op} G_i(P^{op}_t({\boldsymbol{m}}_{1,n}) - e_i P^C_t-f_i P^{k(i)}_t({\boldsymbol{m}}_{1,n}) - x )-\kappa_i \right]dt \nonumber \\ &- K_i \int_{[0,T] \times \overline{\mathcal O}_i} \hat \mu^i_{2,n}(dt, dx) e^{-(\rho+\gamma_i)
  t}  +
  \widetilde K_i \int_{[0,T]\times \mathcal A \times \overline{\mathcal O}_i}  \mu^i_{2,n}(dt, da, dx) e^{-(\rho+\gamma_i) t} \nonumber \\
  &=\int_{[0,T]\times \mathcal A \times \overline{\mathcal O}_i} \widetilde{m}^i_{2,t}(da, dx) e^{-\rho
  t}\lambda_i(a)\left[c_p G_i(P^p_t(\widetilde{{\boldsymbol{m}}}_{1}) - e_i P^C_t-f_i P^{k(i)}_t(\widetilde{{\boldsymbol{m}}}_1) - x )\right.\nonumber \\  & \left.+c_{op} G_i(P^{op}_t(\widetilde{{\boldsymbol{m}}}_1) - e_i P^C_t-f_i P^{k(i)}_t(\widetilde{{\boldsymbol{m}}}_1) - x )-\kappa_i \right]dt \nonumber \\ &- K_i \int_{[0,T] \times \overline{\mathcal O}_i} \widetilde{\hat \mu}^{i}_2(dt, dx) e^{-(\rho+\gamma_i)
  t}  +
  \widetilde K_i \int_{[0,T]\times \mathcal A \times \overline{\mathcal O}_i}  \widetilde{\mu}^{i}_2(dt, da, dx) e^{-(\rho+\gamma_i) t}.
\end{align*}
Recall that, for each $i=1,\ldots,N$, $G_{i}$ is a Lipschitz function with constant $1$. We then obtain
\begin{align*}
&\left|\int_{[0,T]\times \mathcal A \times \overline{\mathcal O}_i} m^i_{2,n,t}(da, dx) e^{-\rho
  t}\lambda_i(a)\left[c_p G_i(P^p_t({\boldsymbol{m}}_{1,n}) - e_i P^C_t-f_i P^{k(i)}_t({\boldsymbol{m}}_{1,n}) - x )\right. \right.\nonumber \\  & \left. \left.+c_{op} G_i(P^{op}_t({\boldsymbol{m}}_{1,n}) - e_i P^C_t-f_i P^{k(i)}_t({\boldsymbol{m}}_{1,n}) - x )-\kappa_i \right]dt \right. \nonumber \\ & \left. - K_i \int_{[0,T]\times \overline{\mathcal O}_i} \hat \mu^i_{2,n}(dt,  dx) e^{-(\rho+\gamma_i)
  t}  +
  \widetilde K_i \int_{[0,T]\times \mathcal A \times \overline{\mathcal O}_i}  \mu^i_{2,n}(dt, da, dx) e^{-(\rho+\gamma_i) t} \right. \nonumber \\
& \left. - \int_{[0,T]\times \mathcal A \times \overline{\mathcal O}_i} \widetilde{m}^i_{2,t}(da, dx) e^{-\rho
  t}\lambda_i(a)\left[c_p G_i(P^p_t(\widetilde{{\boldsymbol{m}}}_{1}) - e_i P^C_t-f_i P^{k(i)}_t(\widetilde{{\boldsymbol{m}}}_1) - x )\right. \right.\nonumber \\  & \left. \left. +c_{op} G_i(P^{op}_t(\widetilde{{\boldsymbol{m}}}_1) - e_i P^C_t-f_i P^{k(i)}_t(\widetilde{{\boldsymbol{m}}}_1) - x )-\kappa_i \right]dt \right. \nonumber \\ & \left. - K_i \int_{[0,T] \times \overline{\mathcal O}_i} \widetilde{\hat \mu}^{i}_2(dt, dx) e^{-(\rho+\gamma_i)
  t}  +
  \widetilde K_i \int_{[0,T]\times \mathcal A \times \overline{\mathcal O}_i}  \widetilde{\mu}^{i}_2(dt, da, dx) e^{-(\rho+\gamma_i) t} \right|
\end{align*}
\begin{align*}
  & \leq  \int_{[0,T]\times \mathcal A \times \overline{\mathcal O}_i} m^i_{2,n,t}(da, dx) e^{-\rho
  t}\lambda_i(a)\Big[c_p |P^p_t({\boldsymbol{m}}_{1,n}) - P^p_t(\widetilde{{\boldsymbol{m}}}_{1})|  \nonumber \\ & +c_{op} |P^{op}_t({\boldsymbol{m}}_{1,n}) - P^{po}_t(\widetilde{{\boldsymbol{m}}}_{1})|+f_i |P^{k(i)}_t({\boldsymbol{m}}_{1,n}) - P^{k(i)}_t({\widetilde{\boldsymbol{m}}}_{1})|\Big] + \nonumber \\ &+\Big|\int_{[0,T]\times \mathcal A \times \overline{\mathcal O}_i} (m^i_{2,n,t}(da, dx)-\widetilde{m}^{i}_{2,t}(da,dx)) e^{-\rho
  t}\lambda_i(a) \\ &\times \Big[c_p G_i(P^p_t(\widetilde{{\boldsymbol{m}}}_{1}) - e_i P^C_t-f_i P^{k(i)}_t(\widetilde{{\boldsymbol{m}}}_1) - x )\nonumber \\  &  +c_{op} G_i(P^{op}_t(\widetilde{{\boldsymbol{m}}}_1) - e_i P^C_t-f_i P^{k(i)}_t(\widetilde{{\boldsymbol{m}}}_1) - x )-\kappa_i \Big]dt \Big| \nonumber \\ 
  &+K_i \Big| \int_{[0,T]\times \overline{\mathcal O}_i} e^{-(\rho+\gamma_i)
  t}(\hat \mu^i_{2,n}(dt, dx)-\widetilde{\hat \mu}^i_{2}(dt, dx)) \Big| \nonumber\\
  &+
  \widetilde K_i \Big|\int_{[0,T]\times \mathcal A \times \overline{\mathcal O}_i} e^{-(\rho+\gamma_i) t} (\mu^i_{2,n}(dt, da, dx)-\widetilde{\mu}^i_{2}(dt, da, dx))\Big|.
\end{align*}
Using similar arguments as in Lemma \ref{BV}, one can show that the maps $t \mapsto h_n(t)$, with $h_n(t)=\int_{ \mathcal A \times \overline{\mathcal O}_i} m^i_{2,n,t}(da, dx) e^{-\rho
  t}\lambda_i(a)$ are of bounded variation, uniformly with respect to $n$, and together with Theorem 3.23 in \cite{ambrosio2000functions}, we obtain that
\begin{align}\label{conv}
\int_{ \mathcal A \times \overline{\mathcal O}_i} m^i_{2,n,\cdot}(da, dx) e^{-\rho
  t}\lambda_i(a) \mapsto \int_{ \mathcal A \times \overline{\mathcal O}_i} \widetilde{m}^i_{2,\cdot}(da, dx) e^{-\rho
  t}\lambda_i(a)
\end{align}
in $L^1([0,T])$.
Furthermore, using again Lemma \ref{conv} item ii. together with the boundedness of the price functionals $P^p$, $P^{op}$ and $P^k$, we derive the convergence to $0$ of the first term from the RHS of the above inequality. The convergence of the third and fourth terms to $0$ is a consequence of the weak convergence of the sequence of measures ${\boldsymbol{m}}_{1,n}$ towards $\widetilde{{\boldsymbol{m}}}_{1}$.
For the second term, we consider a sequence of bounded continuous functions $P_t^{p,m}$ (resp. $P_t^{op,m}$ and $P_t^{k,m}$) approximating in $L^1([0,T])$ the price functionals $P^{p}_t(\widetilde{{\boldsymbol{m}}}_1)$ (resp. $P^{op}_t(\widetilde{{\boldsymbol{m}}}_1)$ and $P^{k,p}_t(\widetilde{{\boldsymbol{m}}}_1)$). Using again the Lipschitz property of the map $G_{i}$, we derive that,
\begin{align*}
&\Big|\int_{[0,T]\times \mathcal A \times \overline{\mathcal O}_i} (m^i_{2,n,t}(da, dx)-\widetilde{m}^{i}_{2,t}(da,dx)) e^{-\rho
  t}\lambda_i(a) \\&\qquad\times \Big[c_p G_i(P^p_t(\widetilde{{\boldsymbol{m}}}_{1}) - e_i P^C_t-f_i P^{k(i)}_t(\widetilde{{\boldsymbol{m}}}_1) - x )\nonumber \\  & \qquad +c_{op} G_i(P^{op}_t(\widetilde{{\boldsymbol{m}}}_1) - e_i P^C_t-f_i P^{k(i)}_t(\widetilde{{\boldsymbol{m}}}_1) - x )-\kappa_i \Big]dt \Big| \nonumber \\ 
& \leq \int_0^T e^{-\rho
  t} \Big|\int_{ \mathcal A \times \overline{\mathcal O}_i} (m^i_{2,n,t}(da, dx)-\widetilde{m}^i_{2,t}(da, dx))\Big| \lambda_i(a)\\&\qquad\times\Big[c_p |P^p_t(\widetilde{{\boldsymbol{m}}_{1}}) - P^{p,m}_t|+c_{op} |P^{op}_t(\widetilde{{\boldsymbol{m}}_{1}}) - P^{op,m}_t|  +f_i |P^{k(i)}_t(\widetilde{{\boldsymbol{m}}_{1}}) - P^{k(i),m}_t|\Big] \nonumber \\ 
& \qquad+\Big|\int_{[0,T]\times \mathcal A \times \overline{\mathcal O}_i} (m^i_{2,n,t}(da, dx)-\widetilde{m}^{i}_{2,t}(da,dx)) e^{-\rho
  t}\lambda_i(a)\\&\qquad\times \Big[c_p G_i(P^{p,m}_t - e_i P^C_t-f_i P^{k(i,m}_t - x )\nonumber \\  & \qquad +c_{op} G_i(P^{op,m}_t - e_i P^C_t-f_i P^{k(i),m}_t - x )-\kappa_i \Big]dt \Big|. 
\end{align*}
The second term in the RHS of the above inequality converges to $0$ in view of the weak convergence of measures, while the convergence  to $0$ of the first term  follows by using again \eqref{conv} and dominated convergence.

\subsection{Proof of Proposition \ref{unique.prop}}
Assume that there are two Nash equilibria, corresponding to families of measures and measure flows
$$
(\hat \mu^i, (\hat m^i_t)_{0\leq t\leq T},\mu^i, ( m^i_t)_{0\leq t\leq T}) \quad \text{and}\quad (\hat {\bar\nu}^i,\hat {\bar\mu}^i, (\hat {\bar m}^i_t)_{0\leq t\leq T}, \bar\nu^i,\bar\mu^i, ( \bar m^i_t)_{0\leq t\leq T})
$$
for $i=1,\dots,N+\overline N$ and price vectors 
$$
P^p,P^{op}, P^1,\dots, P^K \quad \text{and}\quad \overline P^p,\overline P^{op}, \overline P^1,\dots, \overline P^K.
$$
By definition of the Nash equilibrium, we have, for conventional agents, 
\begin{align*}
&\int_{[0,T]\times \mathcal A \times \overline{\mathcal O}_i} m^i_t(da, dx) e^{-\rho
  t}\lambda_i(a)(c_p G_i(P^p_t - e_i P^C_t-f_i P^{k(i)}_t - x )\\ &+c_{op} G_i(P^{op}_t - e_i P^C_t-f_i P^{k(i)}_t - x )-\kappa_i) dt \\ &- K_i \int_{[0,T] \times \overline{\mathcal O}_i} \hat \mu^i(dt, dx) e^{-(\rho+\gamma_i)
  t}
  +
  \widetilde K_i \int_{[0,T]\times \mathcal A \times \overline{\mathcal O}_i}  \mu^i(dt, da, dx) e^{-(\rho+\gamma_i) t}\\
  &\geq \int_{[0,T]\times \mathcal A \times \overline{\mathcal O}_i} \bar m^i_t(da, dx) e^{-\rho
  t}\lambda_i(a)(c_p G_i(P^p_t - e_i P^C_t-f_i P^{k(i)}_t - x )\\ &+c_{op} G_i(P^{op}_t - e_i P^C_t-f_i P^{k(i)}_t - x )-\kappa_i) dt \\ &- K_i \int_{[0,T] \times \overline{\mathcal O}_i} \hat {\bar\mu}^i(dt, dx) e^{-(\rho+\gamma_i)
  t}
  +
  \widetilde K_i \int_{[0,T]\times \mathcal A \times \overline{\mathcal O}_i}  \bar\mu^i(dt, da, dx) e^{-(\rho+\gamma_i) t}
\end{align*}
and 
\begin{align*}
&\int_{[0,T]\times \mathcal A \times \overline{\mathcal O}_i} \bar m^i_t(da, dx) e^{-\rho
  t}\lambda_i(a)(c_p G_i(\overline P^p_t - e_i  P^C_t-f_i \overline P^{k(i)}_t - x )\\ &+c_{op} G_i(\overline P^{op}_t - e_i P^C_t-f_i \overline P^{k(i)}_t - x )-\kappa_i) dt \\ &- K_i \int_{[0,T] \times \overline{\mathcal O}_i} \hat{\bar\mu}^i(dt, dx) e^{-(\rho+\gamma_i)
  t}
  +
  \widetilde K_i \int_{[0,T]\times \mathcal A \times \overline{\mathcal O}_i}  \bar\mu^i(dt, da, dx) e^{-(\rho+\gamma_i) t}\\
  &\geq \int_{[0,T]\times \mathcal A \times \overline{\mathcal O}_i} \bar m^i_t(da, dx) e^{-\rho
  t}\lambda_i(a)(c_p G_i(\overline P^p_t - e_i P^C_t-f_i \overline P^{k(i)}_t - x )\\ &+c_{op} G_i(\overline P^{op}_t - e_i P^C_t-f_i \overline P^{k(i)}_t - x )-\kappa_i) dt \\ &- K_i \int_{[0,T] \times \overline{\mathcal O}_i} \hat {\bar\mu}^i(dt, dx) e^{-(\rho+\gamma_i)
  t}
  +
  \widetilde K_i \int_{[0,T]\times \mathcal A \times \overline{\mathcal O}_i}  \bar\mu^i(dt, da, dx) e^{-(\rho+\gamma_i) t},
\end{align*}
and similarly for the renewable agents. Summing up the two expressions, we get, for conventional agents, 
\begin{multline*}
\int_{[0,T]\times \mathcal A \times \overline{\mathcal O}_i} (m^i_t-\bar m^i_t)(da,dx) e^{-\rho
  t}\lambda_i(a)(c_p G_i(P^p_t - e_i P^C_t-f_i P^{k(i)}_t - x )\\ +c_{op} G_i(P^{op}_t - e_i P^C_t-f_i P^{k(i)}_t - x )-c_p G_i(\overline P^p_t - e_i P^C_t-f_i \overline P^{k(i)}_t - x )\\ -c_{op} G_i(\overline P^{op}_t - e_i P^C_t-f_i \overline P^{k(i)}_t - x )) dt\geq 0, 
\end{multline*}
and for renewable agents,
\begin{multline*}
\int_{[0,T]\times \mathcal A \times \overline{\mathcal O}_i} (m^i_t-\bar m^i_t)(da, dx) e^{-\rho
  t}\bar\lambda_i(a)
  (c_p P^p_t+c_{op}P^{op}_t-c_p \overline P^p_t-c_{op}\overline P^{op}_t) x \,dt\geq 0. 
\end{multline*}
Remark that the terms involving $\mu$ and $\hat \mu$ cancel out in the expressions because they do not depend on the prices. Now, summing up these expressions over all agent types, we get: 
\begin{multline*}
\int_0^T \{G_t(P^p_t, P^{op}_t, P^1_t,\dots,P^K_t) - G_t(\overline P^p_t, \overline P^{op}_t, \overline P^1_t,\dots,\overline P^K_t)\\+ \overline G_t(\overline P^p_t, \overline P^{op}_t, \overline P^1_t,\dots,\overline P^K_t) - \overline G_t(P^p_t, \overline P^{op}_t,  P^1_t,\dots, P^K_t))dt\geq 0,
\end{multline*}
where $G_t$ is the function defined in the proof of Proposition \ref{price.prop}, and $\overline G$ is the same function, but defined with measure flow $\bar m$ instead of $m$. Now, since $(P^p_t, P^{op}_t, P^1_t,\dots,P^K_t)$ is the minimizer of $G_t$, $(P^p_t, \overline P^{op}_t,  P^1_t,\dots, P^K_t)$ is the minimizer of $\overline G_t$ and the two functions are strictly convex, we deduce that the two price vectors are equal, possibly outside a set of zero measure. 
\end{document}